\documentclass{amsart}
\usepackage{amssymb,amsmath}
\usepackage{physics}
\usepackage{amsopn}
\usepackage{xspace}
\usepackage[arrow,matrix,curve]{xy}
\usepackage {color,tikz}
\usepackage{wasysym}
\usepackage{subcaption}
\usepackage{mathtools} %for \eqqcolon, \coloneq
\usepackage{mathrsfs}
\usepackage{algorithm}
\usepackage{algpseudocode}
\usepackage{multirow}
\usepackage[capitalise,nameinlink]{cleveref}

\newcommand{\C}{\mathbb{C}}

\newcommand{\bH}{\mathbb{H}}% ATTENTION DIFFERENT
\newcommand{\bT}{\mathbb{T}}
\newcommand{\N}{\mathbb{N}}

\newcommand{\R}{\mathbb{R}}

\newcommand{\F}{\mathbb{F}}

\newcommand{\bS}{\mathbb{S}}% ATTENTION DIFFERENT

\newcommand\cA{{\ensuremath{\mathcal{A}}}\xspace}
\newcommand\cB{{\ensuremath{\mathcal{B}}}\xspace}

\newcommand\cH{{\ensuremath{\mathcal{H}}}\xspace}

\newcommand\sS{{\ensuremath{\mathscr{S}}}\xspace}
\newcommand\sR{{\ensuremath{\mathscr{R}}}\xspace}
\newcommand\sP{{\ensuremath{\mathscr{P}}}\xspace}

\newcommand\sB{{\ensuremath{\mathscr{B}}}\xspace}

\newcommand{\oh}{{\otimes h}}
\newcommand{\lmd}{{\lambda}}
\newcommand{\eps}{\varepsilon}

\newcommand{\alp}{\alpha}
\newcommand{\lam}{\lambda}

\newcommand{\sig}{\sigma}
\newcommand{\Sig}{\Sigma}

\DeclareMathOperator{\hrank}{hrank}
\DeclareMathOperator{\psdrank}{psdrank}
\DeclareMathOperator{\supp}{supp}
\DeclareMathOperator{\re}{Re}
\DeclareMathOperator{\im}{Im}
\DeclareMathOperator{\Vect}{vec}
\DeclareMathOperator{\inter}{int}

\newtheorem{environment}{Environment}[section]

\newtheorem{remark}[environment]{Remark}
\crefname{remark}{remark}{remarks}
\crefformat{remark}{#2Remark~#1#3}

\crefname{section}{section}{sections}
\crefformat{section}{#2Section~#1#3}

\makeatletter
\if@cref@capitalise
\crefname{subsection}{Subsection}{Subsections}
\crefformat{subsection}{#2Section~#1#3}
\else
\crefname{subsection}{subsection}{subsections}
\crefformat{subsection}{#2Section~#1#3}
\fi
\makeatother

\crefname{equation}{equation}{equations}
\crefformat{equation}{#2(#1)#3}

\crefname{enumi}{enumi}{enumis}
\crefformat{enumi}{#2#1#3}

\crefname{figure}{figure}{figures}
\crefformat{figure}{#2figure~#1#3}

\crefname{table}{Table}{Tables}
\crefformat{table}{#2Table~#1#3}

\newtheorem{alg}[environment]{Algorithm}
\crefname{algorithm}{algorithm}{algorithms}
\crefformat{algorithm}{#2Algorithm~#1#3}

\crefname{notation}{notation}{notations}
\crefformat{notation}{#2Notation~#1#3}

\newtheorem{lemma}[environment]{Lemma}
\crefname{lemma}{lemma}{lemmata}
\crefformat{lemma}{#2Lemma~#1#3}

\crefname{claim}{claim}{claims}
\crefformat{claim}{#2Claim~#1#3}

\newtheorem{corollary}[environment]{Corollary}
\crefname{corollary}{corollary}{corollaries}
\crefformat{corollary}{#2Corollary~#1#3}

\newtheorem{theorem}[environment]{Theorem}
\crefname{theorem}{theorem}{theorems}
\crefformat{theorem}{#2Theorem~#1#3}

\crefname{proposition}{Proposition}{Propositions}
\crefformat{proposition}{#2Proposition~#1#3}

\crefname{conjecture}{Conjecture}{Conjectures}
\crefformat{conjecture}{#2Conjecture~#1#3}

\newtheorem{assumption}[environment]{Assumption}
\crefname{assumption}{Assumption}{Assumptions}
\crefformat{assumption}{#2Assumption~#1#3}

\crefname{fact}{Fact}{Facts}
\crefformat{fact}{#2Fact~#1#3}

\newtheorem{problem}[environment]{Problem}
\crefname{problem}{Problem}{Problems}
\crefformat{problem}{#2Problem~#1#3}

\crefname{question}{Question}{Questions}
\crefformat{question}{#2Question~#1#3}

\theoremstyle{definition}

\newtheorem{definitionx}[environment]{Definition}
\newenvironment{definition}
{\pushQED{\qed}\definitionx}
{}%{\popQED\enddefinitionx}
\crefname{definitionx}{definition}{definitions}
\crefformat{definitionx}{#2Definition~#1#3}
\Crefformat{definitionx}{#2Definition~#1#3}

\newtheorem{examplex}[environment]{Example}
\newenvironment{example}
{\pushQED{\qed}\examplex}
{\popQED\endexamplex}
\crefname{examplex}{example}{examples}
\crefformat{examplex}{#2Example~#1#3}
\Crefformat{examplex}{#2Example~#1#3}

\numberwithin{equation}{section}
\numberwithin{environment}{section}

\definecolor{NiceBlue}{rgb}{0.2,0.2,0.75}

\definecolor{DarkGreen}{rgb}{0,0.65,0}

\title[Separability of Hermitian Tensors and PSD Decompositions]
{Separability of Hermitian Tensors and PSD Decompositions}

\author[Mareike Dressler]{Mareike~Dressler}
\address{Mareike Dressler, Jiawang Nie, and Zi Yang,
Department of Mathematics,
University of California San Diego,
9500 Gilman Drive, La Jolla, CA, USA, 92093.}
\email{mdressler@ucsd.edu, njw@math.ucsd.edu, ziy109@ucsd.edu}

\author[Jiawang Nie]{Jiawang~Nie}

\author[Zi Yang]{Zi~Yang}



\begin{document}

\begin{abstract}
Hermitian tensors are natural generalizations of Hermitian matrices,
while possessing rather different properties.
A Hermitian tensor is separable if it has a Hermitian decomposition
with only positive coefficients, i.e., it is a sum of rank-1 psd Hermitian tensors.
This paper studies how to detect separability of Hermitian tensors.
It is equivalent to the long-standing quantum separability problem
in quantum physics, which asks to tell if a given quantum state is entangled or not.
We formulate this as a truncated moment problem
and then provide a semidefinite relaxation algorithm to solve it.
Moreover, we study psd decompositions of separable Hermitian tensors.
When the psd rank is low, we first flatten them into cubic order
tensors and then apply tensor decomposition methods to compute psd decompositions.
We prove that this method works well if the psd rank is low.
In computation, this flattening approach can detect separability
for much larger sized Hermitian tensors. 
This method is a good start on determining psd ranks of separable Hermitian tensors.
\end{abstract}

\maketitle

\section{Introduction}

Tensors are of tremendous interest in various areas of mathematics and
have broad applications in signal processing, quantum information theory,
machine learning, higher order statistics, and many more.
For an overview on tensors, we refer to \cite{Landsberg:Book,Lim13}.
Let $m>0$ and $n_1, \ldots, n_m >0$ be positive integers.
Denote by $\C^{n_1\times\cdots \times n_m}$ the space of complex tensors of
\textit{order} $m$ and \textit{dimension} $(n_1,\ldots,n_m)$.
A tensor $\cA \in \C^{n_1\times\cdots \times n_m}$ can be represented
as a multi-array $\cA = (\cA_{i_1...i_m} )$,
with labels $i_k\in\{1,...,n_k\}$, $k =1, \ldots, m$.
The \textit{tensor product} of vectors $u_k\in \C^{n_k}, k=1,\ldots, m$,
is their outer product $u_1 \otimes \cdots \otimes u_m$, i.e.,
$(u_1 \otimes \cdots \otimes u_m)_{i_1\cdots i_m}=(u_1)_{i_1}\cdots (u_m)_{i_m}$ 
for all $i_1,\ldots,i_m$ in the range. Tensors of the form
$u_1 \otimes \cdots \otimes u_m$ are called \textit{rank-1} tensors.
Every tensor is a sum of rank-$1$ tensors.
The smallest number of rank-$1$ tensors for decomposing a tensor
$\cA \in \C^{n_1\times\cdots \times n_m}$ is called the
\textit{rank} of $\cA$ and is denoted by $\rank(\cA)$.
The decomposition that achieves the smallest length is called a
\textit{rank decomposition} (see \cite{Landsberg:Book,Lim13}).
When $m \geq 3$, the complexity of determining ranks of tensors is NP-hard
(see \cite{Hillar:Lim}).
We refer to \cite{CLQY20,DeSLim08,Fri16,Landsberg:Book,Lim13}
for related work about ranks of tensors.

\medskip
Hermitian tensors are natural generalizations of Hermitian matrices,
while they exhibit very different properties. A $2m$-order tensor
$\cH \in \C^{n_1\times\cdots\times n_m\times n_1\times\cdots \times n_m}$
is called {\it Hermitian} if for all labels $i_1, ...,  i_m$ and $j_1, ..., j_m$
\[
\cH_{i_1...i_m j_1...j_m} \,=\,  \overline{\cH}_{j_1...j_m i_1...i_m} .
\]
Throughout the paper,
for an array $u$, $\overline{u}$ denotes its complex conjugate.
We first review some basics about Hermitian tensors.

\subsection{Basic properties of Hermitian tensors}

The set of all complex Hermitian tensors in
$\C^{n_1\times\cdots\times n_m\times n_1\times\cdots \times n_m}$
is denoted as $\C^{[n_1, \ldots,  n_m]}$. It is a vector space
of dimension $n_1^2\cdot n_2^2\cdots n_m^2$ over the real field $\R$.
For given vectors $v_i \in \C^{n_i}$, $i=1,\ldots,m$,
denote the rank-$1$ Hermitian tensor
\begin{equation*} \label{Eq:vectorproduct}
[v_1, v_2,\ldots, v_m]_\oh \, := \,  v_1 \otimes v_2 \cdots \otimes v_m \otimes
\overline{v}_1 \otimes \overline{v}_2 \cdots \otimes \overline{v}_m .
\end{equation*}
Every rank-$1$ Hermitian tensor can be expressed as
$\lambda \cdot [v_1, v_2,\ldots, v_m]_\oh$, for a real scalar $\lambda \in \R$.
For every $\cH \in \C^{[n_1, \ldots,  n_m]}$, it is shown in \cite{Ni:HermTensor}
that there exist vectors $u_i^j\in \C^{n_j}$ and real scalars $\lmd_i\in \R$,
$i=1,\ldots,r$, such that
\begin{equation} \label{Eq:rank-oneHD}
\cH \, = \, {\sum}_{i=1}^r \lmd_i \,
[u_i^1, \ldots, u_i^m]_{\otimes_h}.
\end{equation}
The equation \eqref{Eq:rank-oneHD} is called a \textit{Hermitian decomposition}.
The smallest $r$ in (\ref{Eq:rank-oneHD}) is called the \textit{Hermitian rank}
of $\cH$, for which we denote $\hrank(\cH)$. When $r$ is the smallest,
we call \eqref{Eq:rank-oneHD} a \textit{Hermitian rank decomposition} for $\cH$. {In the Hermitian decomposition \eqref{Eq:rank-oneHD},
the magnitude $|\lambda_i|$ can be absorbed into the vector and
only the sign of $\lambda_i$ matters.}
As shown in \cite{NYHerm20}, when $\cH$ is a real Hermitian tensor
(i.e., $\cH$ has only real entries),
it may not have a real Hermitian decomposition
(i.e., the vectors $u_i^j$ in \eqref{Eq:rank-oneHD}
may not be chosen as real vectors).
The subspace of real Hermitian tensors in
$\C^{[n_1, \ldots,  n_m]}$ is denoted as $\R^{[n_1, \ldots,  n_m]}$.

The \emph{inner product} of two Hermitian tensors
$\cA,\cB\in \C^{[n_1, \ldots,  n_m]}$ is
\[
\langle\cA,\cB \rangle \,:= \, \sum_{i_1,\ldots,i_m,j_1,\ldots,j_m} \cA_{i_1,\ldots,i_m,j_1,\ldots,j_m}\overline{\cB}_{i_1,\ldots,i_m,j_1,\ldots,j_m}.
\]
A Hermitian tensor $\cH$ uniquely determines the conjugate-symmetric polynomial
\[
\cH(z,\overline{z}) \, := \, \langle \cH,  [z_1, \ldots, z_m]_{\otimes_h}\rangle,
\]
with $z = (z_1, \ldots, z_m)$ and $z_i \in \C^{n_i}$.
It is interesting to note that $\cH(z,\overline{z})$ achieves only real values when
$\cH$ is Hermitian \cite{Ni:HermTensor}. This inspires us to define
\emph{positive semidefinite}  (psd) Hermitian tensors.
Let $\F = \C$ or $\R$.

\begin{definition}[\cite{NYHerm20}] \label{def:nng}
A Hermitian tensor $\cH \in \F^{[n_1,\ldots,n_m]}$
is called {\it $\F$-positive semidefinite} (\emph{$\F$-psd})
if $\cH(z,\overline{z}) \geq 0$ for all $z_i \in \F^{n_i}$.
Moreover, if $\cH(z,\overline{z}) > 0$ for all $0 \ne z_i \in \F^{n_i}$,
then $\cH$ is called {\it $\F$-positive definite} ($\F$-pd).
The cone of $\F$-psd Hermitian tensors is denoted as
\begin{equation*} \label{psdcone:CR}
\begin{array}{rcl}
	\mathscr{P}_{\F}^{[n_1,\ldots,n_m]} & := &
	\left\{\cH \in \F^{[n_1,\ldots,n_m]}: \,
	\cH(z,\overline{z}) \geq 0 \, \text{ for all } \, z_i \in \F^{n_i}
	\right \}.
	\end{array}
\end{equation*}
\end{definition}

An important property for Hermitian tensors is their separability.

\begin{definition}[\cite{NYHerm20}]   \label{def:SepHermTensor}

A tensor $\cH \in \F^{[n_1, \ldots, n_m]}$
is called {\it $\F$-separable} if
\begin{equation} \label{Eq:separabletensor}
\cH = [u_1^1, \ldots, u_1^m]_{\otimes h} + \cdots +
 [u_r^1, \ldots, u_r^m]_{\otimes h}
\end{equation}
for some vectors $u_i^j\in \F^{n_j}$.
When such a decomposition exists, (\ref{Eq:separabletensor}) is called a
 {\it positive $\F$-Hermitian decomposition},
which we often abbreviate as \textit{positive decomposition}.
The set of $\F$-separable tensors in
$\F^{[n_1, \ldots, n_m]}$ is denoted as $\sS_{\F}^{[n_1,\ldots,n_m]}$.
\end{definition}

{Equivalently, the tensor $\cH$ is $\F$-separable if
every $\lambda_i$ is nonnegative in the decomposition \eqref{Eq:rank-oneHD}.}
The set of $\F$-separable tensors is in fact a cone.
It is dual to the cone of $\F$-psd Hermtian tensors.
We use the superscript $^\star$ to denote the dual cone.
For a cone $C \subseteq \F^{[n_1,\ldots,n_m]}$,
its dual cone is defined as
\[
C^\star := \{ \mathcal{A} \in \F^{[n_1,\ldots,n_m]}:
\, \langle \mathcal{A}, \mathcal{B} \rangle \ge 0 \, \mbox{ for all } B \in C \}.
\]

\begin{theorem}[\cite{NYHerm20}] \label{thm:sep:dual}
The cone $\sS_{\F}^{[n_1,\ldots,n_m]}$
is dual to $\sP_{\F}^{[n_1,\ldots,n_m]}$, i.e.,
\begin{equation*} \label{psd:dual:sep}
\Big( \sS_{\F}^{[n_1,\ldots,n_m]} \Big)^\star =
\sP_{\F}^{[n_1,\ldots,n_m]}, \quad
\Big( \sP_{\F}^{[n_1,\ldots,n_m]} \Big)^\star =
\sS_{\F}^{[n_1,\ldots,n_m]}.
\end{equation*}
\end{theorem}

This paper mostly discusses the case $\F = \C$.
For convenience, $\C$-psd (resp., $\C$-separable)
Hermitian tensors are just simply called psd (resp., separable),
unless the real field $\F=\R$ is considered.

Separable Hermitian tensors can be equivalently expressed by using moments.
If $\cH \in \C^{[n_1,\ldots,n_m]}$ is separable, then it can be written as
\begin{align}\label{Eq:sep:sum}
\cH = {\sum}_{i=1}^r \lmd_i [u_i^1, \ldots, u_i^m]_{\otimes h}
\end{align}
with all $\|u_i^j\|=1,\lambda_i > 0 $.
(Here $\| \cdot \|$ denotes the standard Euclidean norm.)
Let $\mu = \sum_{i=1}^r\lmd_i \delta_{(u_i^1, \ldots, u_i^m)} $
be the weighted sum of Dirac measures,
then \eqref{Eq:sep:sum} is equivalent to that
\begin{align}\label{Eq:sep:int}
    \cH  = \int [z_1,\ldots,z_m]_\oh \mathtt{d}\mu.
\end{align}
The measure $\mu$ is supported in the multi-sphere
\[
\bS_\C^{n_1,\ldots,n_m} \, := \, \{(z_1,\ldots,z_m)\in
\C^{n_1}\times \cdots \times \C^{n_m} \colon \|z_1\|=\cdots=\|z_m\|=1\}.
\]
Conversely, if there is a Borel measure $\mu$ satisfying \eqref{Eq:sep:int},
then $\cH$ must be separable. We have the following result.

\begin{theorem}[\cite{NYHerm20}] \label{thm:Sep:ExistenceMeasure}
A tensor $\cH\in \C^{[n_1, \ldots, n_m]}$ is separable
if and only if there exists a Borel measure $\mu$
such that \eqref{Eq:sep:int} holds and its support
$\supp(\mu)\subseteq \bS_\C^{n_1,\ldots,n_m}$.
\end{theorem}

Detecting separability of a Hermitian tensor is equivalent to checking
the existence of a Borel measure $\mu$ satisfying \eqref{Eq:sep:int}.
This is a truncated moment problem. We will discuss this
with more details in Section~\ref{Sec:Decomposition}.

Hermitian tensors can be naturally flattened into Hermitian matrices \cite{NYHerm20}.
Let $\mathfrak{m} \colon \C^{[n_1, \ldots, n_m]} \to \bH^M$
be the linear map such that
\begin{align}\label{Eq:HermFlatteningMap}
\mathfrak{m}([u_1,\ldots,u_m]_\oh)
\, = \, \big(u_1 u_1^*\big) \boxtimes \cdots \boxtimes \big(u_m u_{m}^*\big).
\end{align}
In the above, $M = n_1 \cdots n_m$, the notation $\bH^M$ denotes the set of all
$M\times M$ Hermitian matrices, the symbol $\boxtimes$
stands for the classical Kronecker product\footnote{
It is more convenient to use the different symbol $\boxtimes$
to distinguish the Kronecker product from the tensor product.
This is because the Kronecker product of two vectors/matrices
has the same order, while the tensor product
of two tensors has a higher order. For instance, it may cause confusion
for defining the matrix flattening $\mathfrak{m}(\cH)$
and the tensor $\mathbb{T}(\cH)$ (see Section \ref{Sec:C-psd-decomp})
if the Kronecker product and tensor product
are denoted by the same notation.
}, and the superscript $^*$
denotes the conjugate transpose.
The matrix $\mathfrak{m}(\cH)$ is called the
\textit{Hermitian flattening matrix} of $\cH$.
Note that $\mathfrak{m}$ gives a bijection between
$\C^{[n_1, \ldots, n_m]}$ and $\bH^M$.
Therefore, a Hermitian tensor can be displayed
by showing its Hermitian flattening matrix.

The separability of $\cH$ can also be equivalently expressed
in terms of the flattening matrix $\mathfrak{m}(\cH)$.
The positive decomposition \eqref{Eq:separabletensor}
is equivalent to
\begin{equation} \label{krondc:m(A)}
\mathfrak{m}(\cH) \,= \, \sum_{i=1}^r
\big(u_i^1 (u_i^1)^*\big) \boxtimes \cdots \boxtimes \big(u_i^m (u_i^{m})^*\big).
\end{equation}
If there exist Hermitian psd matrices
$B_{j}^i \in  \F^{n_j \times n_j}$ such that
\begin{equation}   \label{sepa:A=sum:Bij:ot}
\mathfrak{m}(\cH) \, = \, \sum_{i=1}^s
B_{1}^i \boxtimes \cdots \boxtimes B_{m}^i,
\end{equation}
then $\cH$ must be $\F$-separable.
Similarly, we have $\cH\in \sS_{\F}^{[n_1,\ldots,n_m]}$
if and only if $\cH$ has a decomposition like in \eqref{sepa:A=sum:Bij:ot}.
This observation leads to the following definition
introduced in \cite{NYHerm20}.

\begin{definition}[\cite{NYHerm20}] \label{def:C-psd:dec}
For $\cH\in \sS_{\F}^{[n_1,\ldots,n_m]}$, the \textit{$\F$-psd rank}
of $\cH$, for which we denote $\psdrank_{\F}(\cH)$, is the smallest $s$ 
such that \eqref{sepa:A=sum:Bij:ot} holds
for Hermitian psd matrices $B_{j}^i \in \F^{n_j \times n_j}$.
The equation \eqref{sepa:A=sum:Bij:ot} is called a
\textit{$\F$-psd decomposition} of $\cH$.
\end{definition}

We would like to remark that our notion of $\F$-psd rank is different from
the notion of psd-rank for matrices
that was introduced in \cite{BLP17,deLLG17,LauPio15}.

Hermitian tensors have broad applications in quantum physics 
(see \cite{Doherty:Parrilo:Spedalieri, Gurvits:NP-Hard,LiNiSepa20,Ni:HermTensor}).
It was shown by Gurvits~\cite{Gurvits:NP-Hard} that
the computational complexity of detecting separability is NP-hard.
Therefore, it is a fundamental and attractive task to search for
certificates for separability,
which yield necessary and/or sufficient conditions.
There exist several criteria in the literature 
\cite{Bell,Chen:Albeverio:Fei,Donald:Horodecki:Rudolph,
Guehne:Hyllus:Gittsovich:Eisert,Peres, Horodecki3,Rudolph}.
Despite the range of different criteria in the literature,
most of them give only necessary conditions.
Doherty et al.~\cite{Doherty:Parrilo:Spedalieri}
addressed how to identify entangled states.
Nie and Zhang \cite{NieZhang16} proposed a semidefinite algorithm to
detect separability of real symmetric matrices.
Li and Ni \cite{LiNiSepa20} discussed detecting separability
of general complex Hermitian tensors. They formulated the question
as a truncated moment problem and solved it by Lasserre type moment relaxations.
Hermitian tensor separability is closely related
to the quantum separability problem. We refer to the work \cite{Derksen:Friedland:Lim:Wang:Entanglement,Li:Nakatsukasa:Soma:Uschmajew:Tensors, Ni:Qi:Bai:MeasureOfEntanglement-and-Eigenvalues,
Ni:HermTensor,Qi:Zhang:Ni:HowEntangled}.

\subsection{Contributions}

This article has two major contributions.
We give methods to detect separability of Hermitian tensors
and to compute psd decompositions.

First, we study how to detect separability of a Hermitian tensor.
We formulate this question as a moment optimization problem~\eqref{program:sep-momcone}.
It is an improved version of the one given in \cite{LiNiSepa20}.
Namely, in our new formulation
we choose the leading entry of each decomposing vector to be real nonnegative,
which reduces the number of indeterminate variables of polynomials
and allows us to get tighter relaxations and to solve larger sized problems.
Moreover, our new formulation makes the flat truncation hold for lower order relaxations,
which is the key to obtain a positive decomposition.
The difference to the traditional approach is explained with more details
at the end of Section~\ref{Subsec:MomentProblem}.
We then propose the hierarchy of semidefinite relaxations \eqref{program:sep-SDP}
to solve \eqref{program:sep-momcone}.
Consequently, Algorithm~\ref{Alg:membership-in-sep-cone} is given
to detect separable Hermitian tensors.
In view of quantum entanglements,
this is equivalent to checking separability for quantum states.
The nonseparability can always be detected within finitely many loops
by Algorithm~\ref{Alg:membership-in-sep-cone}.
For separable Hermitian tensors, we show that
the hierarchy of relaxations~\eqref{program:sep-SDP}
%%Algorithm~\ref{Alg:membership-in-sep-cone}
gives a sequence whose accumulation points are optimizers
of the moment optimization \eqref{program:sep-momcone}
(see Theorem~\ref{Thm:AsymptoticConv}).
Furthermore, we prove that
the hierarchy of relaxations~\eqref{program:sep-SDP}
has finite convergence
under certain conditions (see Theorem~\ref{Thm:FiniteConv}).
When the rank condition \eqref{rank:FT:sepaH}
is satisfied for some relaxation order, Algorithm~\ref{Alg:membership-in-sep-cone}
terminates in that loop and produces a positive decomposition.

Second, we study \emph{psd decompositions} of separable Hermitian tensors,
which are expressed as Kronecker products of psd matrices;
see the Definition~\ref{def:C-psd:dec}.
It is mostly an open question to compute
psd decompositions and psd ranks for separable Hermitian tensors
(see~\cite[Problem 7.3]{NYHerm20}).
We make a first step towards solving this question,
when the psd rank is low.
The psd decomposition of a Hermitian tensor $\cH$
is equivalent to the decomposition of a specific non-symmetric tensor $\bT(\cH)$.
If the rank decomposition of $\bT(\cH)$ is in the form of a psd decomposition,
it certifies the separability and also determines the psd rank.
This is stated in {Lemma~\ref{prop: TH}}.
A natural question is how to get a psd decomposition of $\cH$
from the rank decomposition of $\bT(\cH) $.
For the case $m \ge 3$, we show that this is possible
if the psd rank is small;
see Theorems~\ref{Thm:uniqueCpsdDecomposition-smallS} and \ref{prop:r>n2}.
For the case $m=2$, a similar method (see {Theorem}~\ref{prop:m=2})
can be given to get psd decompositions.
This paper is generally not able to determine psd ranks when they are high.
A major advantage of our method for computing psd decompositions
is that it works efficiently for large sized Hermitian tensors
with low psd ranks.
The proposed methods are much faster than in the earlier method
based on solving moment optimization.

The paper is structured as follows.
Section~\ref{Sec:Prelim} introduces the terminology and
reviews basic concepts from polynomial optimization and truncated moment problems.
In Section~\ref{Sec:Decomposition} we discuss how to detect separable Hermitian tensors.
Section~\ref{Sec:C-psd-decomp} studies how to compute psd decompositions
{and certify separability of Hermitian tensors with low psd ranks}.
We close with a discussion of our results and open questions
in Section~\ref{Sec:Conclusion}.

\section{Preliminaries}
\label{Sec:Prelim}

This section introduces the notation and some basics in
polynomial optimization. For more details, we refer to
\cite{Las10,Las15,Laurent:Survey}.

\subsection{Notation}

Throughout the article, we use $\N=\{0,1,2,\ldots\}$
for the set of nonnegative integers, and $\R$ and $\C$
for the field of real numbers and complex numbers, respectively.
For a positive integer, denote $[n] :=\{1,\ldots,n\}$.
For $k=1,\ldots, m$, let $z_k$ be the complex vector variable in $\C^{n_k}$.
The tuple of all such complex variables is denoted by $z=(z_1,\ldots,z_m)$.
Given a real symmetric matrix $M$, the notation $M\; \succeq\; 0$
means that $M$ is positive semidefinite (psd).
The $\bH^{n}$ denotes the set of $n$-by-$n$ complex Hermitian matrices,
and $\bH_+^{n}$ denotes the cone of psd matrices in $\bH^{n}$.
For a matrix $A$, $\text{Row}(A)$ indicates its row space and
$\text{vec}(A)$ denotes its vectorization. 
For a real number $t$, $\lceil t\rceil$ 
denotes the smallest integer that is greater than or equal to $t$. 
Given a real or complex vector $u$, we denote its standard Euclidean norm by $||u||$. 
For a matrix or a vector $a$, we use the notation $a^*$ to denote its conjugate transpose, 
$a^T$ to denote its transpose, while $\overline{a}$ is used for its entry-wise complex conjugate. 
Moreover, we use $a^{\re}$ and $a^{\im}$ its real and imaginary part, respectively. 
The symbol $\otimes$ denotes the tensor product, whereas $\boxtimes$
is used for the classical Kronecker product.
A property for a vector space $S$ is said to hold generically
if it holds everywhere except on a subset of Lebesgue measure zero.

\subsection{Real Algebraic Geometry}

Let $\R[x] := \R[x_1,\ldots,x_n]$ be the ring of real $n$-variate polynomials. 
The set of all $n$-variate polynomials of degree 
less than or equal to $2d$ is written as $\R[x]_{n,2d}$.
If the number of variables is clear from the context, 
we may drop the subscript $n$ and write $\R[x]_{2d}$.  
The degree of a polynomial $p$ is referred to as $\deg(p)$. 
For $\alp=(\alp_1,\ldots,\alp_n)\in \N^n$ denote 
$x^{\alp} = x_1^{\alp_1} \cdots x_n^{\alp_n}$ and $|\alp|=\sum_{i=1}^{n} \alp_i$. 
For a degree $d>0$, we let
\[
\N^n_d=\{\alp \in \N: 0\leq |\alp|\leq d\}
\]
be the set of monomial powers, and $[x]_d$ 
be the vector of all monomials of degrees at most $d$,
ordered in the graded lexicographic ordering, i.e.,
\[
[x]_d=(1,x_1,\ldots,x_n,x_1^2,x_1x_2,\ldots,x_n^d)^T.
\]

For a tuple of polynomials $h = (h_1,\ldots,h_s)\in \R[x]$, the set
\[I(h)=h_1 \cdot \R[x]+\cdots+h_s \cdot \R[x]\]
is the \emph{ideal} generated by $h$. The $k$-th truncation of $I(h)$ 
is the finite dimensional subspace
\[ I_k(h)=h_1\cdot  \R[x]_{k-\deg(h_1)}+\cdots+h_s \cdot \R[x]_{k-\deg(h_s)}.\]

A polynomial $\sig\in \R[x]$ is said to be a 
\emph{sum of squares (SOS)} if $\sig=p_1^2+\cdots+p_t^2$ 
for some real polynomials $p_1,\ldots,p_t$. We use $\Sig[x]$ 
to denote the set of all SOS polynomials in $x$, and $\Sig[x]_{n,d}$ 
to denote the truncation $\Sig[x]\cap \R[x]_{n,d}$. 
Again, we may drop the subscript $n$. The notation $\inter(\Sig[x]_d)$ 
refers to the interior of $\Sig[x]_d$.
Sums of squares can be represented via \emph{semidefinite programming} (SDP);
see \cite{Las10}.
For a tuple $g = (g_1, \ldots, g_t)$, $g_i\in \R[x]$, 
the \emph{quadratic module} generated by $g$ is the set
\[Q(g)=\Sig[x]+ g_1 \cdot \Sig[x]+\cdots+g_t \cdot \Sig[x].\]
Similarly, the $k$-th truncation of $Q(g)$ is
\[
Q_k(g)=\Sig[x]_{2k}+ g_1 \cdot \Sig[x]_{2k-\deg(g_1)}+\cdots+g_t \cdot \Sig[x]_{2k-\deg(g_t)}.
\]

The tuples $h$ and $g$ as above determine the semialgebraic set
\begin{align}\label{Eq:semialgebraic-set-general}
K=\{x\in \R^n: h_1(x)=0,\ldots,h_s(x)=0,\;g_1(x)\geq 0,\ldots,g_t(x)\geq 0\}.
\end{align}
Obviously, if $p\in I(h)+Q(g)$, then $p\geq 0$ on $K$. 
The converse is true under a slightly stronger condition, 
see \cite{Putinar:Positivstellensatz,Nie:OptimalityConditions}.

\subsection{Truncated Moment Problems}
For a given dimension $n$ and degree $d$ let $\R^{\N^n_d}$ be the space of real vectors
that are indexed by $\alp\in \N^n_d$, i.e., $\R^{\N^n_d}=\{y=(y_{\alp})_{\alp\in \N^n_d}: y_{\alp}\in \R\}$. 
It is the space dual to $\R[x]_{n,d}$. A vector in $\R^{\N^n_d}$ 
is called a \emph{truncated multisequence} (tms) of degree $d$. 
We use the notation $y|_d$ to denote the subvector of $y$ whose indices are in $\N^n_d$. 
Each tms gives rise to a linear form acting on $\R[x]_{n,d}$, 
hence for $p=\sum_{\alp \in\N^n_d} p_{\alp} x^{\alp} \in \R[x]_{n,d}$ 
and $y\in \R^{\N^n_d}$ we  define the scalar product
\[\langle p,y\rangle \,= \, \sum_{\alp \in\N^n_d} p_{\alp}y_{\alp}. \]

A Borel measure $\mu$ supported on a set $K\subseteq \R^n$, i.e., 
$\supp(\mu)\subseteq K$, is called a \emph{$K$-measure}. 
If, for a tms $y$, there exists a $K$-measure, such that 
$y_{\alp}=\int x^{\alp}\, \mathtt{d}\mu$ 
for all $\alp \in \N^n_d$, we say $y$ \emph{admits} a $K$-measure $\mu$, 
and we call such a $\mu$ a \emph{$K$-representing} measure for $y$.

Let $p\in \R[x]_{n,2d}$, then the $d$-th \emph{localizing matrix} of $p$, 
generated by a tms $y\in \R^{\N^n_{2d}}$, is the symmetric matrix $L_p^{(d)}(y)$ satisfying
\[\langle pf_1f_2,y \rangle= \Vect({f_1})^T(L_p^{(d)}(y))\Vect({f_2}),\]
for all $f_1,f_2\in \R[x]_{n,d-\lceil \deg(p)/2\rceil}$. 
For a given $p$, $L_p^{(d)}(y)$ is linear in $y$. 
Clearly, if $p(x)\geq 0$ and $y=[x]_{2d}$, 
then $L_p^{(d)}(y)=p(x)[x]_{d-\lceil \deg(p)/2\rceil}[x]_{d-\lceil \deg(p)/2\rceil}^T\succeq 0$.
In the special case, that $p$ is the constant one polynomial $p=1$, 
the localizing matrix $L_1^{(d)}(y)$ reduces to a \emph{moment matrix}, which we denote by
\[M_d(y) \, := \, L_1^{(d)}(y).\]

Let $K$ be as in \eqref{Eq:semialgebraic-set-general}.
A necessary condition for $y$ to admit $K$-measure is
\begin{align}\label{Eq:necessaryCondition-y}
 L_{h_i}^{(d)}(y)=0 \, (1 \le i \le s), \quad
 L_{g_j}^{(d)}(y)\succeq 0 \, (1 \le j \le t), \quad
 M_d(y)\succeq 0.
\end{align}
For convenience, define $L_h^{(d)}(y)$ to be the block diagonal matrix with diagonal blocks 
$L_{h_1}^{(d)}(y), \ldots, L_{h_s}^{(d)}(y)$, 
and similarly $L_g^{(d)}(y)$ denotes the diagonal matrix with blocks $L_{g_j}^{(d)}(y)$. 
Let $t$ be the integer given by $t=\max\{1,\lceil\deg(h)/2\rceil, \lceil\deg(g)/2\rceil\}$.
If $y$, in addition to \eqref{Eq:necessaryCondition-y}, also satisfies the rank condition
\begin{align}\label{Eq:rk-Condition}
\rank M_{d-t}(y)=\rank M_d(y),
\end{align}
then $y$ admits a unique representing $K$-measure and $\mu$ is supported on 
$\rank  M_d(y)$ many distinct points in $K$.  If both \eqref{Eq:necessaryCondition-y} 
and \eqref{Eq:rk-Condition} are satisfied, 
in the literature $y$ is often called \emph{flat with respect to} $h=0$ and $g\geq 0$.

For a tms $w\in \R^{\N^n_k}$ and a degree $d \le k$, the notation
$w\vert_d$ denotes the subvector consisting of entries
$(w)_{\alpha}$ with $|\alpha| \le d$.
For two tms $y \in\R^{\N^n_d}$ and $w\in \R^{\N^n_k}$ with $k>d$,
if $w|_d=y$, we say that $w$ is an \emph{extension} of $y$, or equivalently,
$y$ is a \emph{truncation} of $w$. Clearly, if $y$ is flat and $w|_d=y$,
then $y$ admits a $K$-measure. In such case, we call $w$ a \emph{flat extension}
of $y$ and $y$ a \emph{flat truncation} of $w$.
We refer to \cite{CurFia96,Curto:Fialkow,Lau05}
for the classical flat extension theorem.
Flat extensions and truncations are very useful in
truncated moment problems and optimization
\cite{FiaNie12,Helton:Nie:TMP,Las15,Laurent:Survey,Nie:ATKMP,nie2015linear}.
They are also useful in tensor decompositions~\cite{niegp}.

\section{Detecting separability of Hermitian tensors}
\label{Sec:Decomposition}

The separability of a Hermitian tensor can be detected
by solving a moment optimization problem,
which then can be solved by Lasserre type semidefinite relaxations.
This is done by Li and Ni \cite{LiNiSepa20},
based on the results in \cite{Nie:ATKMP,nie2015linear}.
In this section, we review this method and provide an improved
formulation of the moment optimization.
Furthermore, we prove stronger convergence results.

\subsection{Moment optimization formulation}
\label{Subsec:MomentProblem}

Recall that a Hermitian tensor $\cH \in \C^{[n_1,\ldots,n_m]}$
is separable if and only if there exist vectors $u_i^j \in \C^{n_j}$ such that
\[
\cH \, = \,  [u_1^1, \ldots, u_1^m]_{\otimes h}
+ \cdots + [u_r^1, \ldots, u_r^m]_{\otimes h} .
\]
A complex vector can be written as a sum of its real and imaginary parts.
For $u^j := ((u^j)_1,\ldots,(u^j)_{n_j}) \in \C^{n_j}$, one can write that
\[
u^j \,=\, x^{\re}_j+\sqrt{-1}x^{\im}_j, \quad
x^{\re}_j \in \R^{n_j}, \quad
x^{\im}_j \in  \R^{n_j} .
\]
The coordinates of $x^{\re}_j, x^{\im}_j$ can be labelled as
\[
x^{\re}_j = \big( (x_j^{\re})_1,\ldots,(x_j^{\re})_{n_j} \big), \quad
x^{\im}_j = \big( (x_j^{\im})_{1},\ldots,(x_j^{\im})_{n_j} \big).
\]
It is interesting to note that, for all unitary scalars $\tau_i^j$
(i.e., $|\tau_i^j| = 1$), the above decomposition for $\cH$ is the same as
\begin{equation*}\label{cH=sum:[uji]:tauij}
\cH \, = \, [\tau_1^1 u_1^1, \ldots, \tau_1^m u_1^m]_{\otimes h}
+ \cdots +  [\tau_r^1 u_r^1, \ldots, \tau_r^m u_r^m]_{\otimes h} .
\end{equation*}
For each $u_i^j$, there exists a unitary scalar $\tau_i^j$
such that the first entry of $\tau_i^j u_i^j$ is real and nonnegative, i.e.,
$(x_j^{\re})_1 \ge 0$, $(x_j^{\im})_1=0$.
By Theorem~\ref{thm:Sep:ExistenceMeasure},
a Hermitian tensor $\cH \in \C^{[n_1,\ldots,n_m]}$ is separable if and only if
\[
    \cH  = \int z_1 \otimes \cdots \otimes z_m \otimes
\overline{z}_1 \otimes \cdots \otimes \overline{z}_m \mathtt{d} \mu,
\]
for a Borel measure $\mu$ supported in the multi-sphere $\bS_\C^{n_1,\ldots,n_m}$.
In view of the above observation, such a measure $\mu$
can be further chosen to be supported in the set
\begin{align*} \label{set:sC:real>=0}
\bS_{\C,+}^{n_1,\ldots,n_m}
 :=\left\{(u^1,\ldots,u^m)
%\left|
: \begin{array}{l}
 u^j \in \C^{n_j}, \, \|u^j\| = 1, \, (x_j^{\re})_1 \ge 0, (x_j^{\im})_1=0
 \end{array} %\right.
 \right\}.
\end{align*}
For convenience of notation, for each $j=1,\ldots,m$, we denote that
\[
x_j := (x_j^{\re}, x_j^{\im})=\big( (x_j^{\re})_1,\ldots,(x_j^{\re})_{n_j}, \,
 (x_j^{\im})_{2},\ldots,(x_j^{\im})_{n_j} \big) \in \R^{2n_j-1}.
\]
For neatness of labelling, we also write that
\[
x_j := \big( (x_j)_1,\ldots,(x_j)_{n_j}, \,
 (x_j)_{n_j+1},\ldots,(x_j)_{2n_j-1} \big) .
\]
Then $\bS_{\C,+}^{n_1,\ldots,n_m}$ can be equivalently written
as the semialgebraic set
\begin{align}
\label{Eq:Def:K}
K \, :=\, \left\{(x_1,\ldots,x_m)
:\begin{array}{l}
x_j \in \R^{2n_j-1}, \, \| x_j \| = 1,  \, (x_j)_1 \ge 0.
 \end{array}
 \right\}.
\end{align}
Let $\sB(K)$ denote the set of all Borel measures supported in $K$ and
\begin{align}
\label{Eq:Def:x}
x \,:= \, (x_1,\ldots,x_m)\in \R^{2N-m},
\end{align}
with $N :=\sum_{j=1}^{m} n_j$. The set $K$ can be equivalently given as
\begin{align*}
K=\{x\in \R^{2N-m}: h(x)=0, g(x)\geq 0 \},
\end{align*}
where $h := \big(\|x_1\|^2-1,\ldots,\|x_m\|^2-1 \big)$
and  $g(x):=\big( (x_1)_1,\ldots, (x_m)_1 \big)$.

Next, we consider the label set
\begin{align} \label{labelSET:S}
S \,:= \, \{(i_1,\ldots,i_m): i_1\in [n_1],\ldots,i_m\in [n_m]\}.
\end{align}
Its cardinality is $M=n_1\cdots n_m$. For two labeling tuples in $S$
\[
I:=(i_1,\ldots,i_m), \quad J:=(j_1,\ldots,j_m),
\]
we define the ordering
$I < J$ if the first nonzero entry of $I-J$ is negative.
{We remark that any ordering on multi-indices
works here. We just choose the  one above for convenience.}
For $I<J$, let $P_{IJ}$ denote the polynomial
\begin{align}
\label{Eq:P_IJ}
P_{IJ}(x) \,:=\, \prod_{s=1}^{m}
\big(x_s^{\re}+\sqrt{-1}x_s^{\im}\big)_{i_s} \cdot
\big(x_s^{\re}-\sqrt{-1}x_s^{\im}\big)_{j_s}.
\end{align}
Therefore, the positive Hermitian decomposition~\eqref{Eq:sep:int}
is equivalent to that
\begin{align}\label{Eq:sep:int:poly}
\cH_{IJ} = \int_{K} P_{IJ}(x) \mathtt{d}\mu,
\quad \text{ for all } I,J\in S,
\end{align}
for a Borel measure $\mu$ supported in $K$. Then
Theorem~\ref{thm:Sep:ExistenceMeasure} implies the following.

\begin{corollary}\label{cor:sep-meas}
A tensor $\cH\in \C^{[n_1, \ldots, n_m]}$ is separable if and only if 
there exists a measure $\mu\in \sB(K)$ such that 
\eqref{Eq:sep:int:poly} is satisfied.
\end{corollary}

Next, we formulate the above as a moment optimization problem,
following similar ideas in \cite{LiNiSepa20}
adapted to our new formulation of the moment problem.
We write each $P_{IJ}$ as a sum of real and imaginary parts
\[
P_{IJ}(x) \, = \,R_{IJ}(x)+\sqrt{-1}T_{IJ}(x)
\]
for real polynomials $R_{IJ}, T_{IJ}\in \R[x] =\R[x_1,\ldots,x_m]$.
Likewise, the tensor entries $\cH_{IJ}$ of $\cH$ can be written as
\begin{align}\label{Eq:Hre-Him}
\cH_{IJ} \, = \,
\cH_{IJ}^{\re}+\sqrt{-1}\cH_{IJ}^{\im},
\end{align}
for real entries $\cH_{IJ}^{\re}, \cH_{IJ}^{\im}$.
Since $\cH$ is Hermitian, it holds that
\[
\cH_{IJ}^{\re}=\cH_{JI}^{\re}, \quad
\cH_{IJ}^{\im}=-\cH_{JI}^{\im}, \quad \cH_{II}^{\im}=0 .
\]
Therefore, it suffices to consider
$\cH_{IJ}^{\re}$ with $I\le J$ and $\cH_{IJ}^{\im}$ with $I< J$.
For a polynomial $F(x) \in \R[x]$,
we consider the moment optimization problem
\begin{align}
\label{program:sep-meas}
\begin{cases}
\underset{\mu}{\text{ min }} &\int_{K} F(x) \mathtt{d}\mu\\
\text{ s.t. } &\cH_{IJ}^{\re} = \int_{K} R_{IJ}(x) \mathtt{d}\mu,
        \;\; (I\leq J), \\
\phantom{\text{ s.t. }} &\cH_{IJ}^{\im} = \int_{K} T_{IJ}(x) \mathtt{d}\mu,
          \;\; (I<J), \\
\phantom{\text{ s.t. }} &\mu\in \sB(K).
\end{cases}
\end{align}
To ensure that \eqref{program:sep-meas} has a unique minimizer, one can choose
$F(x)$ to be a generic polynomial in $\Sig[x]_{2m}$.
We introduce the moment cone
\begin{align} \label{Eq:momentcone}
\sR_{2m}(K) \, := \, \left\{ %y\in \R^{\N^{2N-m}_{2m}}:
y = (y_\alpha)
%\left|
:\begin{array}{l}
\exists \mu\in \sB(K), \, \text{ such that } \\
(y)_{\alpha}=\int x^{\alpha} \mathtt{d}\mu, \,
\forall \alpha \in \N^{2N-m}_{2m}
\end{array}
%\right.
\right\}.
\end{align}
Then, \eqref{program:sep-meas} is equivalent to the following optimization
\begin{align}\label{program:sep-momcone}
\begin{cases}
\underset{y}{\text{ min }} &\langle F,y \rangle\\
\text{ s.t. } &\cH_{IJ}^{\re} = \langle R_{IJ},y \rangle,
     \;\; (I\le J ), \\
\phantom{\text{ s.t. }} &\cH_{IJ}^{\im} = \langle T_{IJ},y \rangle,
      \;\; (I<J ),  \\
\phantom{\text{ s.t. }} & y\in \sR_{2m}(K). % \sR^{2N-m}_{2m}.
\end{cases}
\end{align}
For the coefficient vector $\textbf{f} := (f^{\re}, f^{\im})$ with
\[
f^{\re} \, := \,  \big( f_{IJ}^{\re}  \big)_{I \le J}, \quad
f^{\im} \, := \,   \big( f_{IJ}^{\im}  \big)_{I < J},
\]
denote the polynomials
\[
G(\textbf{f}) \, := \, F(x) -\sum_{I\le J} f_{IJ}^{\re} \cdot R_{IJ}(x)
- \sum_{I< J}f_{IJ}^{\Im} \cdot T_{IJ}(x).
\]
Then the optimization problem dual to \eqref{program:sep-momcone} is
\begin{align} \label{program:sep-SDP-dual-meas}
\left\{ \begin{array}{cl}
\max\limits_{\textbf{f}} & \sum_{I\le J} f^{\re}_{IJ}\cH^{\re}_{IJ}
             +\sum_{I< J}f^{\im}_{IJ}\cH^{\im}_{IJ} \\
\text{s.t.} & G(\textbf{f}) \in \mathscr{P}_{2m}(K),
\end{array} \right.
\end{align}
where $\mathscr{P}_{2m}(K)$ denotes the cone of polynomials in
$\R[x]_{2m}$ that are nonnegative on $K$.

In the recent work \cite{LiNiSepa20}, a similar moment optimization
is formulated for detecting separability.
We point out the differences between the one in \cite{LiNiSepa20}
and the one presented in the above.
In \cite{LiNiSepa20}, the optimization is formulated for the set
$\bS_{\C}^{n_1,\ldots,n_m}$, while we formulate it for the set $\bS_{\C,+}^{n_1,\ldots,n_m}$.
Consequently, the variable $x$ in \cite{LiNiSepa20}
has length $2N$, while the $x$ here only has length $2N-m$.
Moreover, in \cite{LiNiSepa20} the polynomial $F$
is chosen to have degree $2m+2$,
while  we require $F$ only to be of degree $2m$.
The description of the set $\bS_{\C,+}^{n_1,\ldots,n_m}$
has $m$ more scalar inequalities (i.e., $g(x) \ge 0$) than $\bS_{\C}^{n_1,\ldots,n_m}$,
but it has $m$ less polynomial indeterminate variables.
The computational cost of solving moment relaxations
rapidly grows as the number of indeterminate variables increases.
Moreover, the inequality $g(x) \ge 0$ makes the moment relaxation stronger,
because it gives additional localizing matrix inequalities.
Therefore, our moment optimization formulation (\ref{program:sep-momcone})
is more efficient for the computational purpose,
which is also demonstrated in our numerical experiments.
In contrast, the classical moment optimization formulation in \cite{LiNiSepa20}
is less efficient. Please note that when $\cH$ is separable,
there are always infinitely many decompositions such as
\[
  \cH \, = \,  [\tau_1^1u_1^1, \ldots, \tau_1^mu_1^m]_{\otimes h}
+ \cdots + [\tau_r^1u_r^1, \ldots, \tau_r^m u_r^m]_{\otimes h},
\]
because the above is satisfied for all unitary scalars $\tau_i^j$.
Consequently, the flat truncation condition is unlikely
to be satisfied for moment relaxations.
However, our moment optimization formulation
can avoid this issue by requiring the leading entry
of each decomposing vector to be real and nonnegative.

\subsection{A semidefinite relaxation algorithm}
\label{Subsec:Semidefinite Algorithm}

The moment cone $\sR_{2m}(K)$ can be approximated well by semidefinite relaxations.
Select a generic $F(x)\in \Sigma[x]_{2m}$.
Consider the hierarchy of semidefinite relaxations
\begin{align}
\label{program:sep-SDP}
\begin{cases}
\underset{w}{\text{ min }} &\langle F,w \rangle\\
\text{ s.t. } &\langle R_{IJ},w \rangle=\cH_{IJ}^{\re},
        \, (I,J\in S, I\le J )\\
\phantom{\text{ s.t. }}& \langle T_{IJ},w\rangle=\cH_{IJ}^{\im},
       \, (I,J\in S, I<J) \\
\phantom{\text{ s.t. }} &L_h^{(k)}(w)=0, M_k(w)\succeq 0,L_g^{(k)}(w) \succeq 0,\\
\phantom{\text{ s.t. }} &w\in \R^{\N^{2N-m}_{2k}},
\end{cases}
\end{align}
for relaxation orders $k = m, m+1, \cdots$.
The dual optimization of the above is
\begin{align}
\label{program:sep-SDP-dual}
\begin{cases}
\max\limits_{\textbf{f}} & \sum_{I\le J}f^{\re}_{IJ}\cH^{\re}_{IJ}
                +\sum_{I< J}f^{\im}_{IJ}\cH^{\im}_{IJ} \\
\text{ s.t. } & G(\textbf{f}) \in I_{2k}(h)+Q_k(g).
\end{cases}
\end{align}
This yields the following algorithm.

\begin{alg} \label{Alg:membership-in-sep-cone}
\textsc{Detecting separability for Hermitian tensors.}
\begin{itemize}

\item [Input:] A Hermitian tensor $\cH \in \C^{[n_1, \ldots, n_m]}$.

\item [Output:]
Either a positive $\C$-Hermitian decomposition of $\cH$, particularly affirming membership in $\sS_{\C}^{[n_1,\ldots,n_m]}$, or an answer that $\cH$ is not separable.

\item [Step 0:] Let $k=m$. Choose a generic $F(x)\in \Sigma[x]_{2m}$.

\item [Step 1:] Solve the semidefinite optimization~\eqref{program:sep-SDP}.
If it is infeasible, output that $\cH$ is not separable, and stop;
otherwise, solve it for a minimizer $w^{\star,k}$ and let $t := 1$.

\item [Step 2:]
Let $w:=w^{\star,k}|_{2t}$. Check whether or not the rank condition
\begin{equation} \label{rank:FT:sepaH}
\rank M_{t-1}(w) \, = \, \rank M_t(w)
\end{equation}
holds. If it does, go to Step~4; otherwise, go to Step~3.

\item [Step 3:]
If $t<k$, set $t=t+1$ and go to Step~2; otherwise, set $k=k+1$ and go to Step~1.

\item [Step 4:]
Let $r := \rank M_t(w)$. Compute the weights $\lam_1>0,\ldots,\lam_r>0$
and $v^{(1)},\ldots,v^{(r)}\in K$ such that
\begin{equation} \label{decomp:w:lam:x(i)}
w \, = \, \lam_1 [v^{(1)}]_{2t} + \cdots + \lam_r [v^{(r)}]_{2t}.
\end{equation}
For each $i=1,\ldots,r$, write that
$v^{(i)} = (v_1^{(i)}, \ldots, v_m^{(i)})$ with each $v_j^{(i)} \in \R^{2n_j-1}$
and for $j=1,\ldots, m$,  let
\[
u_i^{j} \,:= \, \big( (v_j^{(i)})_1, \ldots,  (v_j^{(i)})_{n_j} \big) +
\sqrt{-1}  \big( 0, (v_j^{(i)})_{n_j+1}, \ldots,  (v_j^{(i)})_{2n_j-1} \big).
\]
Output the positive decomposition
$\cH = \sum_{i=1}^{r} \lam_i [u_i^1, \ldots, u_i^m]_{\otimes h}.$
\end{itemize}
\end{alg}

In the Step 0, the generic polynomial
$F  \in \Sigma[x]_{2m}$ can be selected as $F = [x]^T_m (G^TG)[x]_m$,
with a random square matrix $G$ of length $\tbinom{2N-m+d/2}{d/2} $,
i.e., each entry of $G$ is a real random variable fulfilling normal (Gaussian) distribution.
The Step 1 is justified by Theorem~\ref{Thm:Non-Separability}
in Subsection~\ref{Subsec:Convergence}.
The Step 2 requires checking if $w$ satisfies the rank condition \eqref{rank:FT:sepaH}.
When the rank condition~\eqref{rank:FT:sepaH} is satisfied,
one can use the method in \cite{Henrion:Lasserre:GloptiPoly-OptimizationSolutionextraction}
to get a positive $\C$-Hermitian decomposition in \eqref{decomp:w:lam:x(i)}.
This method is implemented in the software \texttt{GloptiPoly3} \cite{Henrion:Lasserre:GloptiPoly3}.
We point out that the vectors $u_i^{j}$ must belong to the set
$\bS_{\C,+}^{n_1,\ldots,n_m}$ if \eqref{rank:FT:sepaH} holds
{(see
\cite{Henrion:Lasserre:GloptiPoly-OptimizationSolutionextraction,Laurent:Survey})}.
Algorithm~\ref{Alg:membership-in-sep-cone}
can be conveniently implemented in \texttt{GloptiPoly3};
see Subsection~\ref{Subsec:Examples-sep} for numerical experiments.
As shown in Theorem \ref{Thm:AsymptoticConv},
the hierarchy of relaxations~\eqref{program:sep-SDP}
asymptotically converges for solving \eqref{program:sep-meas}.
Moreover, Theorem~\ref{Thm:FiniteConv} shows that
it terminates within finitely many loops
under certain conditions.
In our numerical experiments, the rank condition~\eqref{rank:FT:sepaH}
is satisfied for all cases.

Algorithm~\ref{Alg:membership-in-sep-cone} is similar to the Algorithm~1
in \cite{LiNiSepa20}. They are both based on solving the Lasserre type
Moment-SOS relaxations. However, they are also quite different.
The constraining set $\bS_{\C,+}^{n_1,\ldots,n_m}$ 
in Algorithm~\ref{Alg:membership-in-sep-cone} is a subset
of $\bS_{\C}^{n_1,\ldots,n_m}$ which is used in \cite{LiNiSepa20}.
Thus $\bS_{\C,+}^{n_1,\ldots,n_m}$ has $m$ fewer variables
and the semidefinite relaxation \eqref{program:sep-SDP}
is stronger than the one in \cite{LiNiSepa20}. Moreover,
the objective polynomial $F$ has a lower degree than the one in \cite{LiNiSepa20}.
Therefore, Algorithm~\ref{Alg:membership-in-sep-cone}
can detect separability for larger sized Hermitian tensors.
We remark that Algorithm~\ref{Alg:membership-in-sep-cone}
can be applied to check separability for all Hermitian tensors,
no matter their ranks are high or low.

\subsection{Convergence properties}
\label{Subsec:Convergence}

Now we study the convergence of Algorithm~\ref{Alg:membership-in-sep-cone}.
In \cite[Theorem~2]{LiNiSepa20}, Li and Ni proved the subsequent properties
for their semidefinite relaxations:
(I) If the semidefinite relaxation is infeasible for some order $k$,
then the Hermitian tensor $\cH$ is not separable.
(II) If $\cH$ is separable, then their relaxations can asymptotically get
a positive Hermitian decomposition,
i.e., the accumulation points of minimizers of the relaxations
solve the moment optimiation problem \eqref{program:sep-meas}
and give positive decompositions.
Their proof uses the results in \cite{Nie:ATKMP}.
In this subsection, we prove stronger convergence properties
for Algorithm~\ref{Alg:membership-in-sep-cone}.
In fact, if $\cH$ is not separable, we show that
the semidefinite relaxation \eqref{program:sep-SDP}
must be infeasible for all $k$ large enough.
Furthermore, we prove the finite convergence for
Algorithm~\ref{Alg:membership-in-sep-cone} under some conditions.

First, we show that non-separability of a Hermitian tensor
is equivalent to infeasibility of
the semidefinite relaxation \eqref{program:sep-SDP}
for some order $k$.

\begin{theorem}  \label{Thm:Non-Separability}
Let $\cH$, $\cH_{IJ}^{\re}$, $\cH_{IJ}^{\im}$ be as in \eqref{Eq:Hre-Him}.
Then,  $\cH$ is not separable (i.e., $\cH\notin \sS_{\C}^{[n_1,\ldots,n_m]}$)
if and only if the semidefinite relaxation~\eqref{program:sep-SDP}
is infeasible for some $k$.
\end{theorem}

\begin{proof}
``if" direction: Note that \eqref{program:sep-SDP} 
is a relaxation of \eqref{program:sep-momcone}. If \eqref{program:sep-SDP} is infeasible,
then \eqref{program:sep-momcone} must be infeasible
and hence $\cH$ is not separable.

\smallskip \noindent
``only if" direction:
Recall that $\sP_{\C}^{[n_1,\ldots,n_m]}$ is the dual cone of
$\sS_{\C}^{[n_1,\ldots,n_m]}$, by Theorem~\ref{thm:sep:dual}.
If $\cH$ is not $\C$-separable, there exists a psd tensor
$\cA_1\in \sP_{\C}^{[n_1,\ldots,n_m]}$ such that $\langle \cA_1,\cH\rangle < 0$.
For $\eps > 0$, let $\cA$ be the Hermitian tensor such that
\[ \cA(z,\overline{z})=\cA_1(z,\overline{z})+\eps (z_1^*z_1)\cdots (z_m^*z_m) . \]
If $\eps > 0$ is sufficiently small, $\langle \cA,\cH\rangle<0$
and $\cA$ is $\C$-positive definite. Write that
$\cA = \cA^{\re} + \sqrt{-1}\cA^{\im}$, where $\cA^{\re},\cA^{\im}$
are both real tensors. Since $\cA$ is positive definite,
for the variable $x$ as in the Subsection~\ref{Subsec:MomentProblem}, we have that
\[
\cA(x) \,:= \, \langle \cA, [x_1^{\re}+\sqrt{-1}x_1^{\im}, \ldots, 
x_m^{\re}+\sqrt{-1}x_m^{\im}]_{\otimes h} \rangle =
\]
\[
\sum_{I,J\in S}\cA^{\re}_{IJ} R_{IJ} +
\sum_{I,J\in S}\cA^{\im}_{IJ} T_{IJ}>0, \text{ for all } x \in K.
\]
Select $\textbf{f} = (f^{\re},f^{\im})$ as follows
\[
f^{\re}_{IJ} = \left \{
 \begin{array}{ll}
   -\cA^{\re}_{IJ} & \text{if $I=J$} \\
  -2\cA^{\re}_{IJ} & \text{if $I<j$}
 \end{array} \right. ,\, \,
f^{\im}_{IJ} = -2\cA^{\im}_{IJ} \text{ for $I<J$}.
\]
Thus, $G(\textbf{f}) = F(x) + \cA(x)$.
By Putinar's Positivstellensatz \cite{Putinar:Positivstellensatz},
we have $\cA(x)\in I_{2k_0}(h)+Q_{k_0}(g) $ for some $k_0$.
Since $F(x) \in \Sig[x]_{2m}$, we have
\[
F(x)+\tau\cA(x) \, \in \, I_{2k_0}(h)+Q_{k_0}(g)
\]
for all $\tau>0$. This implies that $\tau \textbf{f}$ is feasible for
\eqref{program:sep-SDP-dual} for all $\tau>0$.
Moreover, for the above choice of $\textbf{f}$,
the objective value in \eqref{program:sep-SDP-dual} is such that
\[
\sum_{I\le J}\tau f^{\re}_{IJ}\cH^{\re}_{IJ}+
\sum_{I< J}\tau f^{\im}_{IJ}\cH^{\im}_{IJ}
=\tau \langle -\cA,\cH\rangle \to +\infty,
\]
as $\tau \to +\infty$. Therefore, the dual problem \eqref{program:sep-SDP-dual}
is unbounded from above and hence, by duality,
the primal problem~\eqref{program:sep-SDP}
must be infeasible for all $k \ge k_0$.
\end{proof}

Second, we prove the asymptotic convergence
of the hierarchy of relaxations \eqref{program:sep-SDP}
for solving the moment optimization \eqref{program:sep-meas}.
For the minimizer $w^{\star,k}$, recall that the notation
$w^{\star,k}\vert_{2m}$ denotes the subvector of entries
$(w^{\star,k})_{\alpha}$ with $|\alpha| \le 2m$.
The $w^{\star,k}\vert_{2m}$ is called
the truncation of $w^{\star,k}$ with degree $2m$.
The asymptotic convergence for Algorithm~\ref{Alg:membership-in-sep-cone}
means that the truncated sequence $\{w^{\star,k}\vert_{2m}\}^\infty_{k=m}$
of minimizers is bounded and all its accumulation points are optimizers
of the moment optimization \eqref{program:sep-meas}.
The proof is based on results in \cite{nie2015linear}.

\begin{theorem}\label{Thm:AsymptoticConv}
Let $\cH$ be a separable Hermitian tensor.
If $F(x)$ is a generic polynomial in $\Sig[x]_{2m}$,
then we have the following properties:
\begin{enumerate}
		
\item[(i)] For all $k\geq m$, the semidefinite relaxation
\eqref{program:sep-SDP} has an optimizer $w^{\star,k}$.

\item[(ii)] The truncated sequence  $\{w^{\star,k}\vert_{2m}\}_{k=m}^\infty$
is bounded and all its accumulation points are optimizers of
the moment optimization problem \eqref{program:sep-meas}.

\end{enumerate}
\end{theorem}
\begin{proof}
Since the Hermitian tensor $\cH$ is separable
(i.e., $\cH\in \sS_{\C}^{[n_1,\ldots,n_m]}$),
there is a measure $\mu$ satisfying \eqref{program:sep-meas},
by Corollary~\ref{cor:sep-meas}. Hence the problem
\eqref{program:sep-momcone} is feasible.

(i) Since \eqref{program:sep-momcone} is feasible,
the problem \eqref{program:sep-SDP} is feasible as well.
The genericity of $F(x)$ implies that $F$ lies in the interior of
$\Sig[x]_{2m}$. Therefore, \eqref{program:sep-SDP} is bounded from below and
$(f^{\re},f^{\im})=(0,0)$ is an interior point of
the dual optimization \eqref{program:sep-SDP-dual}.
Therefore, the strong duality holds and the semidefinite relaxation
\eqref{program:sep-SDP} must have an optimizer $w^{\star,k}$.

(ii) The set $K$ satisfies the ball condition
\[ \|x\|^2 = \sum_{j=1}^{m}\|x_j^{\re}\|^2+\|x_j^{\im}\|^2\leq m, \]
so the archimedeanness holds for the constraining polynomials of $K$.
The conclusion then follows from
\cite[Theorem 4.3(ii)]{nie2015linear}.
\end{proof}

Last, we study when Algorithm~\ref{Alg:membership-in-sep-cone},
terminates within finitely many loops.
This occurs under some assumptions on the optimizer of
\eqref{program:sep-SDP-dual-meas}.

\begin{assumption} \label{assump:finite converge}
Suppose $\textbf{f}^*$ is a maximizer of the optimization
\eqref{program:sep-SDP-dual-meas} and the polynomial
$F^*:= G(\textbf{f}^*)$ satisfies the conditions:
\begin{enumerate}
\item [i)] There exists $k_1$ such that $F^*\in I_{2k_1}(h)+Q_{k_1}(g)$;

\item [ii)] The optimization problem
 \begin{equation*}
				\min \, F^*(x) \,\, s.t. \,\, h(x)=0, g(x)\ge 0
 \end{equation*}
has finitely many KKT points $u$ for which $F^*(u)=0$.
\end{enumerate}
\end{assumption}

We refer to \cite{NieFT} for the notion of KKT points.
Assumption \ref{assump:finite converge} holds if $F^*$
is a generic point on the boundary of $\mathscr{P}_{2m}(K)$
(see \cite{Nie:OptimalityConditions}).
The following is the finite convergence result.

\begin{theorem}\label{Thm:FiniteConv}
Let $\cH\in \sS_{\C}^{[n_1,\ldots,n_m]}$.
Suppose $F(x) \in \inter(\Sig[x]_{2m})$,
Assumption \ref{assump:finite converge} holds,
and $w^{\star,k}$ is a minimizer of \eqref{program:sep-SDP}
for the relaxation order $k$. Then, for all $k>t$ sufficiently large,
the rank condition \eqref{rank:FT:sepaH} must be satisfied.
\end{theorem}

\begin{proof}
The conclusion follows from Theorem~4.6 of \cite{nie2015linear}.
\end{proof}

\subsection{Numerical examples}
\label{Subsec:Examples-sep}

In this subsection, we present examples for detecting separability of Hermitian tensors
by using Algorithm~\ref{Alg:membership-in-sep-cone}.
The algorithm can be implemented in the software 
\texttt{GloptiPoly3} \cite{Henrion:Lasserre:GloptiPoly3}, 
which calls the SDP solver 
\texttt{SeDuMi}~\cite{Henrion:Lasserre:GloptiPoly-OptimizationSolutionextraction}.
Since the semidefinite programs are solved numerically,
we display only four decimal digits for the computational results.
The computation is implemented in MATLAB R2019b,
on an \verb!Intel(R) Core(TM) i7-8550U CPU!
with $3.79$ GHz and $16$ GB of RAM.
{In Examples \ref{Ex:Comparison1} and \ref{Ex:Comparison2},
we also compare our new moment optimization formulation
with the traditional one in \cite{LiNiSepa20}.}

\begin{example}
Consider the Hankel tensor $\cH \in \C^{[2,2]}$ in \cite{Nie:Ye:HankelTensor}
such that
\[
\cH_{i_1i_2j_1j_2} \, = i_1+i_2+j_1+j_2
\]
for all $1\leq i_1,i_2, j_1,j_2\leq 2$. The tensor $\cH$ is not separable,
detected by Algorithm~\ref{Alg:membership-in-sep-cone},
since the semidefinite relaxation \eqref{program:sep-SDP}
is infeasible for $k=2$. The computation took around $0.8$ second.
\end{example}

\begin{example}
Consider the tensor $\cH \in \C^{[3,3]}$ such that
\[
\cH_{i_1i_2j_1j_2} \,=\, i_1j_1 + i_2j_2
\]
for all $i_1, i_2, j_1, j_2$ in the range. It is separable, detected by
Algorithm~\ref{Alg:membership-in-sep-cone} for $k=2$.
We got the positive Hermitian decomposition
$\cH= \lam_1 [u_1^1,u_1^2]_{\otimes h} + \lam_2[u_2^1,u_2^2]_{\otimes h}$,
with weights $\lam_1=\lam_2=42$ and
\begin{align*}
u_1^1=\begin{pmatrix}
\sqrt{14}/14\\ \sqrt{14}/7\\ 3/\sqrt{14}\\
\end{pmatrix}, u_1^2=\begin{pmatrix}\sqrt{3}/3\\\sqrt{3}/3\\\sqrt{3}/3\end{pmatrix}, u_2^1=\begin{pmatrix}\sqrt{3}/3\\\sqrt{3}/3\\\sqrt{3}/3\end{pmatrix}, u_2^2=\begin{pmatrix}
\sqrt{14}/14\\\sqrt{14}/7\\ 3/\sqrt{14}\\
\end{pmatrix}.
\end{align*}
The computation took around $2.7$ seconds.
\end{example}

\begin{example}
Consider the Hermitian tensor
$\cH=\frac{1}{2} \psi_1 \otimes
\overline{\psi_1}+\frac{1}{2} \psi_2\otimes \overline{\psi_2},
$
where
\[
\begin{array}{l}
\psi_1:=\frac{1}{\sqrt{3}}(e_1\otimes e_1+e_1\otimes e_2+\sqrt{-1}e_2\otimes e_2), \\ \psi_2:=\frac{1}{3\sqrt{2}}(e_1\otimes e_1-e_1\otimes e_2+4\sqrt{-1}e_2\otimes e_1),
\end{array}
\]
for $e_1:=(1,0),e_2:=(0,1)$.
In terms of the eigenvalue decomposition
of the Hermitian flattening matrix,
it was shown in \cite[Example 6.1]{Ni:HermTensor}
that this state is not separable.
The semidefinite relaxation \eqref{program:sep-SDP} is infeasible for $k=2$,
so we know $\cH \notin \sS_{\C}^{[2,2]}$ not separable.
The computation took around $0.8$ second.
\end{example}

In what follows, we consider more general Hermitian tensors.
The weights $\lambda_i$
are set to be one by scaling the vectors $u_i^j$ accordingly.
That is, we display the positive Hermitian decomposition as
$\cH=\sum_{i=1}^{r} [u_i^1, \ldots, u_i^m]_{\otimes h}$.
Moreover, we use the notation $\mathrm{i} := \sqrt{-1}$.
Note that a Hermitian tensor $\cH$
can be equivalently represented by
its Hermitian flattening matrix $\mathfrak{m}(\cH)$.

\begin{example}\label{Ex:Comparison1}
Consider $\cH \in \sS_{\C}^{[2,2]}$
with the Hermitian flattening matrix
\begin{align*}
\mathfrak{m}(\cH)=
\begin{pmatrix*}[r] 32 & -8+4\sqrt{-1} & 5+\sqrt{-1} & -4\\
-8-4\sqrt{-1} & 32 & 2-8\sqrt{-1} & 1-3\sqrt{-1}\\
5-\sqrt{-1} & 2+8\sqrt{-1} & 28 & -8+5\sqrt{-1}\\
-4 & 1+3\sqrt{-1} & -8-5\sqrt{-1} & 27
\end{pmatrix*}.
\end{align*}
By Algorithm \ref{Alg:membership-in-sep-cone} with $k=3$,
we got $\cH=\sum_{i=1}^7 [u_i^1,u_i^2]$ where 
$U_1:=[u_1^1,\ldots,u_7^1]^T$, $U_2:=[u_1^2,\ldots,u_7^2]^T$
are respectively {\scriptsize
\[
U_1 = \begin{pmatrix*}[r]
			0.0000 &  -0.6229 - 1.2207\mathrm{i}\\
   			1.2181 &  -1.6690 - 0.6625\mathrm{i}\\
   			0.0000 &  -0.4866 - 2.1243\mathrm{i}\\
   			1.7133 &  -0.8874 - 1.1767\mathrm{i}\\
   			2.1137 &  -0.1468 + 0.1672\mathrm{i}\\
   			0.5438 &   0.3833 + 1.4707\mathrm{i}\\
   			2.0208 &   0.2600 + 0.1676\mathrm{i}\\
\end{pmatrix*}, \,
U_2 = \begin{pmatrix*}[r]
			0.9511 &  -0.9660 + 0.2009\mathrm{i}\\
			1.4188 &  -1.5025 - 0.6616\mathrm{i}\\
			1.7916 &   0.8825 + 0.8722\mathrm{i}\\
			1.5836 &   1.4979 - 0.5967\mathrm{i}\\
			1.8678 &   1.0045 - 0.1405\mathrm{i}\\
			1.1035 &  -1.1751 + 0.0825\mathrm{i}\\
			1.1819 &  -1.4087 + 0.8934\mathrm{i}\\
\end{pmatrix*}.
\]}\noindent The computation took about $2$ seconds.
This Hermitian tensor is separable.
The classical formulation in \cite{LiNiSepa20}
took about $120$ seconds
to solve for the same relaxation order $k=3$
and did not get a positive decomposition.
For $k=4$, the one in \cite{LiNiSepa20} took about $3.5$
hours and still failed to detect separability.
\end{example}

\begin{example}\label{Ex:Comparison2}
Consider $\cH \in \sS_{\C}^{[3,3]}$ whose flattening matrix
$\mathfrak{m}(\cH)$ is {\small
\[
\begin{pmatrix*}[r]
10 & -2-2{}\mathrm{i} & 1+1{}\mathrm{i} & 7-\mathrm{i} & -2-4{}\mathrm{i} & 2{}\mathrm{i} & -4-6{}\mathrm{i} & 0 & -2\\ -2+2{}\mathrm{i} & 10 & -6+1{}\mathrm{i} & -2 & 5+3{}\mathrm{i} & -5+1{}\mathrm{i} & -4-4{}\mathrm{i} & -4+2{}\mathrm{i} & 3+1{}\mathrm{i}\\
1-\mathrm{i} & -6-\mathrm{i} & 12 & 2+4{}\mathrm{i} & -5-\mathrm{i} & 8+1{}\mathrm{i} & 4+6{}\mathrm{i} & -3-\mathrm{i} & -4-2{}\mathrm{i}\\ 7+1{}\mathrm{i} & -2 & 2-4{}\mathrm{i} & 9 & -1-3{}\mathrm{i} & -1-\mathrm{i} & 1-7{}\mathrm{i} & -2 & -4+2{}\mathrm{i}\\
 -2+4{}\mathrm{i} & 5-3{}\mathrm{i} & -5+1{}\mathrm{i} & -1+3{}\mathrm{i} & 8 & -5-\mathrm{i} & -2{}\mathrm{i} & 4 & 2{}\mathrm{i}\\
 -2{}\mathrm{i} & -5-\mathrm{i} & 8-\mathrm{i} & -1+1{}\mathrm{i} & -5+1{}\mathrm{i} & 11 & 2+4{}\mathrm{i} & -2+4{}\mathrm{i} & 3-5{}\mathrm{i}\\
 -4+6{}\mathrm{i} & -4+4{}\mathrm{i} & 4-6{}\mathrm{i} & 1+7{}\mathrm{i} & 2{}\mathrm{i} & 2-4{}\mathrm{i} & 20 & -3-\mathrm{i} & 2\\
  0 & -4-2{}\mathrm{i} & -3+1{}\mathrm{i} & -2 & 4 & -2-4{}\mathrm{i} & -3+1{}\mathrm{i} & 17 & -9+1{}\mathrm{i}\\ -2 & 3-\mathrm{i} & -4+2{}\mathrm{i} & -4-2{}\mathrm{i} & -2{}\mathrm{i} & 3+5{}\mathrm{i} & 2 & -9-\mathrm{i} & 22
\end{pmatrix*}.
\]}\noindent
By Algorithm~\ref{Alg:membership-in-sep-cone} with $k=2$,
we got the positive Hermitian decomposition
$\cH=\sum_{i=1}^{9} [u_i^1,u_i^2]_{\otimes h}$, where $U_1:=[u_1^1,\ldots,u_9^1]^T,U_2:=[u_1^2,\ldots,u_9^2]^T$
are given as {\scriptsize
\[
U_1 = \begin{pmatrix*}[r]
		0.0000  & -1.1632 - 0.5687{}\mathrm{i}&  -0.2972 - 0.8659{}\mathrm{i} \\
		-0.0000 & -1.1309 + 0.5564{}\mathrm{i}&  -0.8436 - 0.2873{}\mathrm{i} \\
		 1.3161 & 0.6580 -0.6580{}\mathrm{i} & -0.6580 - 0.6580{}\mathrm{i} \\
		 1.0000 &  -1.0000 + 1.0000{}\mathrm{i} &  1.0000 - 1.0000{}\mathrm{i}\\
		 1.4142 &   0.7071 + 0.7071{}\mathrm{i} & -1.4142 + 0.0000{}\mathrm{i}\\
		 0.8691 &  -0.4346 - 0.4345{}\mathrm{i} &  0.8691 - 0.0000{}\mathrm{i}\\
		 1.3375 &  -0.6687 - 0.6687{}\mathrm{i} &  0.0000 + 1.3375{}\mathrm{i}\\
		 1.3375 &  -0.6687 + 0.6687{}\mathrm{i} &  1.3375 + 0.0000{}\mathrm{i}\\
		 0.6921 &  -0.3460 - 0.3461{}\mathrm{i} &  0.6921 + 0.0000{}\mathrm{i}
\end{pmatrix*},
\]
\[
U_2 = \begin{pmatrix*}[r]
		0.7928  &  0.0000 - 0.7928{}\mathrm{i} & -0.7928 - 0.7929{}\mathrm{i}\\
		0.7718  & -0.0000 - 0.7718{}\mathrm{i} & -0.7718 - 0.7718{}\mathrm{i}\\
		1.0746  &  0.0000 - 0.0000{}\mathrm{i} & -1.0746 + 1.0746{}\mathrm{i}\\
		1.4142  & 0.7071 - 0.7071{}\mathrm{i}  &-0.0000 - 1.4142{}\mathrm{i}\\
		1.0000  &  1.0000 + 1.0000{}\mathrm{i} & -1.0000 + 1.0000{}\mathrm{i}\\
		0.0000  &  0.0000 + 0.0000{}\mathrm{i} & -0.5884 + 1.2419{}\mathrm{i}\\
		1.0574  &  1.0574 - 0.0000{}\mathrm{i} &  1.0574 + 1.0574{}\mathrm{i}\\
		1.0574  &  1.0574 - 0.0000{}\mathrm{i} & -1.0574 + 1.0574{}\mathrm{i}\\
		0.0000  &  0.0000 - 0.0000{}\mathrm{i} &  0.8382 + 0.7034{}\mathrm{i}
\end{pmatrix*} .
\]}\noindent
The computation took around $4.2$ seconds.
This Hermitian tensor is separable.
The formulation in \cite{LiNiSepa20} took about
$30$ seconds to solve for $k=2$ and failed to
get a positive decomposition.
For $k=3$, the one in \cite{LiNiSepa20} took about $6$ hours
to solve and still failed to detect separability.
\end{example}

\begin{example}
Consider the tensor $\cH \in \sS_{\C}^{[2,2,2]}$
with $\mathfrak{m}(\cH)$ being the matrix
{\scriptsize
\[
\begin{pmatrix*}[r]
18 & -2+8{}\mathrm{i} & -16-8{}\mathrm{i} & 4-4{}\mathrm{i} & -2-4{}\mathrm{i} & 2+16{}\mathrm{i} & 2-2{}\mathrm{i} & -4-12{}\mathrm{i}\\ -2-8{}\mathrm{i} & 50 & 16{}\mathrm{i} & -34-26{}\mathrm{i} & -10+20{}\mathrm{i} & 2+32{}\mathrm{i} & 8-16{}\mathrm{i} & 30-30{}\mathrm{i}\\
-16+8{}\mathrm{i} & -16{}\mathrm{i} & 32 & -18+18{}\mathrm{i} & 10+6{}\mathrm{i} & -8-16{}\mathrm{i} & -2-6{}\mathrm{i} & 6+14{}\mathrm{i}\\
4+4{}\mathrm{i} & -34+26{}\mathrm{i} & -18-18{}\mathrm{i} & 78 & 4-28{}\mathrm{i} & 2+2{}\mathrm{i} & -14+30{}\mathrm{i} & -4+22{}\mathrm{i}\\
-2+4{}\mathrm{i} & -10-20{}\mathrm{i} & 10-6{}\mathrm{i} & 4+28{}\mathrm{i} & 22 & 12+6{}\mathrm{i} & -18-8{}\mathrm{i} & -12\\ 2-16{}\mathrm{i} & 2-32{}\mathrm{i} & -8+16{}\mathrm{i} & 2-2{}\mathrm{i} & 12-6{}\mathrm{i} & 70 & -16+12{}\mathrm{i} & -50-26{}\mathrm{i}\\
2+2{}\mathrm{i} & 8+16{}\mathrm{i} & -2+6{}\mathrm{i} & -14-30{}\mathrm{i} & -18+8{}\mathrm{i} & -16-12{}\mathrm{i} & 30 & 2+10{}\mathrm{i}\\
 -4+12{}\mathrm{i} & 30+30{}\mathrm{i} & 6-14{}\mathrm{i} & -4-22{}\mathrm{i} & -12 & -50+26{}\mathrm{i} & 2-10{}\mathrm{i} & 86
\end{pmatrix*}.
\]}\noindent
By Algorithm~\ref{Alg:membership-in-sep-cone} with $k=3$,
we got the positive Hermitian decomposition \linebreak
$\cH=\sum_{i=1}^{6} [u_i^1,u_i^2, u_i^3]_{\otimes h}$, where
\[
U_1:=[u_1^1,\ldots,u_6^1]^T, \,
U_2:=[u_1^2,\ldots,u_6^2]^T, \,
U_3:=[u_1^3,\ldots,u_6^3]^T
\] are shown as follows
{\scriptsize
\[
U_1	= \begin{pmatrix*}[r]
		1.0191 &   2.0381 - 0.0000{}\mathrm{i}\\
		1.2222 &  -1.2222 - 1.2222{}\mathrm{i}\\
	   -0.0000 &  -1.2100 - 1.6470{}\mathrm{i}\\
	   -0.0001 &   1.2959 - 0.4964{}\mathrm{i}\\
		1.4837 &  -0.7419 - 0.7418{}\mathrm{i}\\
		1.3077 &  -1.3077 + 0.0000{}\mathrm{i}
\end{pmatrix*}, \quad
U_2 = \begin{pmatrix*}[r]
		1.6113 &  -1.6113 + 0.0000{}\mathrm{i}\\
		0.9467 &  -1.8935 + 0.0000{}\mathrm{i}\\
		1.4451 &   0.0000 + 1.4451{}\mathrm{i}\\
		0.9812 &   0.0000 + 0.9812{}\mathrm{i}\\
		1.4836 &  -0.7418 + 0.7418{}\mathrm{i}\\
		0.8270 &   0.0000 + 1.6541{}\mathrm{i}	
\end{pmatrix*},
\]
\[
U_3 = \begin{pmatrix*}[r]
		1.2181 &  -0.6090 - 1.8270{}\mathrm{i}\\
		1.7285 &  -0.8642 + 0.8642{}\mathrm{i}\\
		1.4451 &   0.8671 - 1.1561{}\mathrm{i}\\
		0.9812 &   0.5888 - 0.7851{}\mathrm{i}\\
		1.2848 &  -0.0000 + 1.2850{}\mathrm{i}\\
		1.3077 &   1.3077 - 0.0000{}\mathrm{i}
	\end{pmatrix*} .
\]}\noindent
The computation took around $5$ minutes.
This Hermitian tensor is separable.
\end{example}

\section{The psd Decompositions}
\label{Sec:C-psd-decomp}

This section studies positive semidefinite (psd) decompositions (see Definition \ref{def:C-psd:dec})
for separable Hermitian tensors. Let $\F = \C$ or $\R$.
If a Hermitian tensor $\cH$ is $\F$-separable, then
there are vectors $u_i^j \in \F^{n_j}$ such that $\cH$ has the decomposition
\[
\cH = [u_1^1, \ldots, u_1^m]_{\otimes h} + \cdots +
[u_r^1, \ldots, u_r^m]_{\otimes h}.
\]
Recall that equivalently,
the tensor $\cH$ is $\F$-separable if and only if
it has a $\F$-psd decomposition like
\[
  \mathfrak{m}(\cH) \, = \, {\sum}_{i=1}^s
B_1^i \boxtimes \cdots \boxtimes B_m^i,
\]
where $B_j^i \in  \F^{n_j \times n_j}$ are Hermitian psd matrices.

When $\cH$ is $\R$-Hermitian decomposable
(i.e., \eqref{Eq:rank-oneHD} holds for real vectors $u_i^j$),
we would like to remark that
$\cH$ is $\R$-separable if and only if it is $\C$-separable
(see \cite[Lemma~6.2]{NYHerm20}).
The $\R$-separability of $\R$-Hermitian decomposable tensors
can also be detected by checking their $\C$-separability.
Therefore, this section focuses on the complex case $\F=\C$.
For convenience of writing, the $\C$-psd ranks and
$\C$-psd decompositions are just simply called psd ranks and psd decompositions.
Similarly, the psd rank $\psdrank_{\C}(\cH)$
is also abbreviated to $\psdrank(\cH)$.

It is generally a big challenge to compute psd ranks, as well as psd decompositions.
%directly from Hermitian flattening matrices.
To address this problem, we appeal to the theories and methods for tensor decompositions.
A Hermitian tensor $\cH \in \C^{[n_1,\ldots,n_m]}$
can be flattened into the $m$th order tensor
$\bT(\cH) \in \C^{n_1^2 \times  \cdots \times n_m^2}$ such that
\begin{equation} \label{def:T(H)}
\bT(\cH) = \sum_{i=1}^r \big(u_i^1 \boxtimes \overline{ u_i^1 } \big)
\otimes \cdots \otimes \big(u_i^m  \boxtimes \overline{ u_i^m }\big).
\end{equation}
The psd decomposition \eqref{sepa:A=sum:Bij:ot}
yields the following decomposition
\begin{align}\label{Eq:DecompositionT}
\bT(\cH) = \sum_{i=1}^s
\text{vec}({B}^i_1) \otimes \cdots \otimes \text{vec}({B}^i_m),
\end{align}
where $\text{vec}({B}^i_j)$ denotes the vectorization of the matrix $B_j^i$.
Decomposing $\cH$ directly is usually very hard,
since its Hermitian rank can be very high. However,
the psd rank of $\cH$ may be smaller than $\hrank(\cH)$.
For Hermitian tensors with small psd ranks, 
we show in Theorems~\ref{Thm:uniqueCpsdDecomposition-smallS} 
and \ref{prop:r>n2} that the corresponding decomposition \eqref{Eq:DecompositionT}
of $\bT(\cH)$ is generically
the unique rank decomposition of the tensor $\bT(\cH)$.
For the case of small $\psdrank(\cH)$, decomposing $\bT(\cH)$
offers an alternate efficient way for computing
psd decompositions, which also certifies the separability.

The following result shows that when the tensor $\bT(\cH)$
admits a rank decomposition in the form of \eqref{Eq:DecompositionT},
then $\cH$ must be separable and its psd rank coincides with $\rank(\bT(\cH))$.

\begin{lemma}\label{prop: TH}
For a Hermitian tensor $\cH\in \C^{[n_1,\ldots,n_m]}$,
if $\rank (\bT(\cH) ) = s$ and it has the decomposition \eqref{Eq:DecompositionT}
for Hermitian psd matrices ${B}^i_j$,
then $\cH$ is separable and $\psdrank(\cH) = s$.
\end{lemma}
\begin{proof}
The above decomposition of $\bT(\cH)$ admits a psd decomposition of
$\cH$ with the same length, so we have $\cH$ is separable and
$\psdrank(\cH)\le s $. Each psd decomposition of $\cH$
also gives a decomposition of the tensor $\bT(\cH)$. Thus,
\[
s \, = \, \rank(\bT(\cH)) \,\le \, \psdrank(\cH) \, \le \,  s .
\]
Hence $\cH$ is separable and $\psdrank(\cH)=s$.
\end{proof}

\subsection{The case $m\ge 3$}

Lemma~\ref{prop: TH} connects the psd decomposition of
$\cH$ and the rank decomposition of $\bT(\cH)$.
We are interested in incidents when the rank decomposition of
$\bT(\cH)$ gives a psd decomposition for $\cH$.
These decompositions are related as in \eqref{Eq:DecompositionT}.
If there is a unique rank decomposition for $\bT(\cH)$,
we are able to get a psd decomposition for $\cH$, as well as the psd rank.
Notably, the psd rank can be smaller than the Hermitian rank
and the tensor $\bT(\cH)$ has higher individual dimensions than $\cH$.
There exist efficient tensor decomposition methods
for tensors with high individual dimensions and low ranks.
This is the motivation for us to consider the flattening tensor $\bT(\cH)$.

For the case $m \ge 3$, if the psd rank is low, the tensor
$\bT(\cH)$ has the unique decomposition of the same rank.
The classical Kruskal's theorem concerns uniqueness of tensor decompositions.
Here we briefly review this result.
For a set $S$ of vectors, its \textit{Kruskal rank}, denoted as $k_S$,
is defined to be the maximum number $k$ such that every subset of $k$
vectors in $S$ is linearly independent.
The following result is due to Kruskal for $m=3$ 
\cite{Kruskal:UniqueDecomposition,Kruskal:RankDecompositionUniqueness}
and is due to Sidiropoulos and Bro \cite{Sidiropoulos:Bro} for $m>3$.

\begin{theorem} \label{thm:krus}
Let $\cA\in  \C^{n_1\times \cdots \times n_m} $ be the tensor
\[
\cA = \sum_{i=1}^r \lambda_i u_i^1 \otimes \cdots \otimes u_i^m,
\]
{
where $u_i^j \in \C^{n_j}$ and the $\lambda_i$'s are nonzero complex scalars.}
Let $U_j = [u_1^j,\ldots,u_r^j] $ and $k_{U_j}$ be the Kruskal rank of $U_j$. If
\[
	k_{U_1}+\cdots + k_{U_m} \ge 2r+m-1,
\]
then $\rank(\cA) = r$ and its rank decomposition is unique,
up to scaling and permutations.
\end{theorem}

Theorem~\ref{thm:krus} gives a uniqueness result about
the tensor decomposition of $\bT(\cH)$.
Recall that $\bH^{n}$ denotes the set of $n \times n$ complex Hermitian matrices.
The set $\bH^{n}$ is a vector space of dimension $n^2$ over the real field.
The following is a simple but useful fact.

\begin{lemma} \label{lemma: LID herm}
Let $s\le n^2$. If $M_1,\ldots,M_s$ are generic
in $\bH^{n}$, then they are linearly independent
over both fields $\R$ and $\C$.
\end{lemma}
\begin{proof}
Let $\varphi: \bH^{n} \to \R^{n^2}$ be the linear map such that
\begin{equation}\label{eq: Herm to R}
\varphi(H) \, := \,
\begin{bmatrix}  \varphi_1(H) \\ \varphi_2(H) \end{bmatrix},
\end{equation}
where $\varphi_1(H)$ is the vectorization of the real part of 
the upper triangular part of $H$ and $\varphi_2(H)$ 
is the vectorization of the imaginary part of the \textit{strictly} upper triangular part of $H$.
If the matrices $M_1,\ldots,M_s$ are such that
\[
\det\left(\begin{bmatrix} \varphi(M_1) & \cdots & \varphi(M_s) \end{bmatrix}^T
\begin{bmatrix} \varphi(M_1) & \cdots & \varphi(M_s) \end{bmatrix} \right)
\ne 0,
\]
then $\varphi(M_1),\ldots,\varphi(M_s)$
are linearly independent in $\R^{n^2}$
and hence $M_1,\ldots,M_s$ are also linearly independent in $\bH^{n}$.
The non-vanishing of the above determinant is a generic condition.
Therefore, if $M_1,\ldots,M_s$ are generic in $\bH^{n}$,
then they must be linearly independent, over both $\R$ and $\C$.
\end{proof}

Recall that $\bH_+^{n}$ denotes the cone of psd matrices in $\bH^{n}$.
A property is said to hold \textit{generically} in $\bH_+^{n}$ (resp., in  $\bH^{n}$)
if it holds everywhere in $\bH_+^{n}$ (resp., in  $\bH^{n}$) 
except a subset of $\bH_+^{n}$ (resp., in  $\bH^{n}$) 
with Lebesgue measure zero.

\begin{theorem}  \label{Thm:uniqueCpsdDecomposition-smallS}
Let $m\ge 3$ and $n_1\ge \cdots \ge n_m\ge 2$.
Suppose the Hermitian tensor $\cH\in \sS_\C^{[n_1,\ldots,n_m]}$
has the psd decomposition
\[
\mathfrak{m}(\cH) \, = \, {\sum}_{i=1}^s B_1^i \boxtimes \cdots \boxtimes B_m^i,
\]
for psd matrices $B_j^i \in \bH_+^{n_j}$. For each $j=1, \ldots, m$,
let $B_j := \{ B^1_j,\ldots, B^s_j \}$.
\begin{enumerate}
\item[(i)] If
\[
k_{B_1}+\cdots + k_{B_m} \ge 2s+m-1,
\]
then $\rank(\bT(\cH)) = s $ and $\bT(\cH) $ has the unique rank decomposition
\[
\bT(\cH)  \, = \, {\sum}_{i=1}^s
\Vect(B_1^i) \otimes \cdots \otimes \Vect(B_m^i).
\]
\item[(ii)] If $s\le n_2^2$ and each $B_j^i$ is generic in
$\bH_+^{n_j}$,
then the above conclusion is also true.
\end{enumerate}
\end{theorem}
\begin{proof}
(i) This follows directly from Theorem~\ref{thm:krus}.

(ii) When $s=1$, it must hold that $\rank\,\bT(\cH) =1$.
The rank decomposition of every rank-1 tensor is unique, so the conclusion holds.
So we may consider the case $s\ge2$.
Each $B_j^i$ is a $n_j\times n_j$ Hermitian matrix.
By Lemma~\ref{lemma: LID herm},
$k_{B_j} = \min(s,n_j^2) $ when $B^1_j,\ldots, B^s_j$
are generic. Since $n_2^2 \ge s \ge 2$ and each $n_j \ge 2$,
\[
k_{B_1}+\cdots + k_{B_m} = \min(s,n_1^2) + \min(s,n_2^2)+\cdots +\min(s,n_m^2)
\]
\[
\ge 2s + \min(s,n_3^2) + \cdots +\min(s,n_m^2)  \ge 2s + 2(m-2) \ge 2s+m-1.
\]
Therefore, the second statement follows from the first one.
\end{proof}

Theorem~\ref{Thm:uniqueCpsdDecomposition-smallS}
generalizes the result that if a tensor
$\cA\in \C^{n_1\times \cdots\times n_m} $
has a decomposition of length $r \le n_2$, then $\rank(\cA)=r $ for the generic case.
If $n_1\ge r >n_2$, this property may not hold 
\cite{domanov2013,domanov2015generic,stegeman2010uniqueness}.
The subsequent theorem provides a sufficient condition
for the rank decomposition to be unique, for the case $r\ge n_2$.
For two matrices $U = [u_1,\ldots,u_\ell], V=[v_1,\ldots,v_\ell]$
of the same number of columns,
their \emph{Khatri-Rao product} \cite{KR product}, 
denoted by $\odot$, is defined such that
\[
U \odot V \, := \,
\big[ u_1 \boxtimes v_1 ,\ldots, u_\ell \boxtimes v_\ell \big].
\]
For a matrix $X \in \C^{n \times r}$, define
$\mathcal{C}(X)\in \C^{n(n-1)/2 \times r(r-1)/2}$ 
to be the compound matrix \cite{domanov2013}
consisting of $2\times 2$ minors of $X$.
The rows of $\mathcal{C}(X)$ are labelled by pairs of
two distinct rows of $X$, and the columns
are labelled by pairs of two distinct columns of $X$.

\begin{theorem}[\cite{domanov2013,de2006link}] \label{thm:r>n2}
Let $\cA \in \C^{n_1\times n_2 \times n_3}$ be such that
\[
\cA = \lambda_1 a_1 \otimes b_1 \otimes c_1 + \cdots
+\lambda_r a_r\otimes b_r \otimes c_r,
\]
and let $A = [a_1,\ldots,a_r],B=[b_1,\ldots,b_r],C=[c_1,\ldots,c_r]$.
If $\rank(A) = r $ and
$\mathcal{C}(B) \odot \mathcal{C}(C)$ has linearly independent columns, then
$\rank(\cA) = r$ and the rank decomposition of $\cA$
is unique, up to scaling and permutation.
\end{theorem}

The above theorem can be applied to $\bT(\cH)$,
then we get the following result.

\begin{theorem}\label{prop:r>n2}
Let $n_1\ge n_2 \ge n_3\ge 2$. Suppose
$\cH \in \sS_\C^{[n_1,n_2,n_3]}$ is given as
\[
\mathfrak{m}(\cH) \, = \, {\sum}_{i=1}^s
A_{i}  \boxtimes B_{i} \boxtimes C_{i},
\]
for psd matrices
$A_i \in \bH^{n_1}_+,B_i\in \bH^{n_2}_+,C_i\in \bH^{n_3}_+$. Assume
\[
s\le n_1^2 \quad \quad \and \quad
\frac{s(s-1)}{2} \le \frac{n_2^2(n_2^2-1)}{2} \cdot \frac{n_3^2(n_3^2-1)}{2} .
\]
If $A_i$, $B_i$, $C_i$ are generic in
$\bH_+^{n_1},  \bH_+^{n_2}, \bH_+^{n_3}$ respectively,
then $\rank(\bT(\cH)) = s $ and $\bT(\cH)$ has the unique rank decomposition
\[
\bT(\cH) \, = \, {\sum}_{i=1}^s
\Vect(A_i) \otimes \Vect(B_i) \otimes \Vect(C_i).
\]	
\end{theorem}
\begin{proof}
Let $\tilde{A}:=[\text{vec}(A_1),\ldots,\text{vec}(A_s)]$,
$
\tilde{B}:=[\text{vec}(B_1),\ldots,\text{vec}(B_s)]$, 
and $\tilde{C}:=[\text{vec}(C_1),\ldots,\text{vec}(C_s)].
$
Since $\tilde{A}\in \C^{n_1^2 \times s} $,
Lemma~\ref{lemma: LID herm} yields that for generic psd matrices
$\{A_i\}_{i=1}^s$, we have $\rank(\tilde{A}) = s $ when $s\le n_1^2 $.

In addition, we have $\mathcal{C}(\tilde{B})\odot \mathcal{C}(\tilde{C}) \in \C^{\frac{n_2^2(n_2^2-1)}{2}\frac{n_3^2(n_3^2-1)}{2} \times \frac{s(s-1)}{2}} $.
By the assumption,
\[
\frac{s(s-1)}{2} \le \frac{n_2^2(n_2^2-1)}{2}\frac{n_3^2(n_3^2-1)}{2} .
\]
Let $\varphi:\mathbb{H}^n \to \R^{n^2}$
be the map defined in \eqref{eq: Herm to R}.
Denote $\hat{B}:=[\varphi(B_1),\ldots,\varphi(B_s)]$ 
and $\hat{C}:=[\varphi(C_1),\ldots,\varphi(C_s)] $. 
Let $\tilde{R}_i$ and $\hat{R}_i$ be the $i$th row of $\tilde{B} $ and $\hat{B}$ respectively. 
Row vectors $\tilde{R}_i^{\re},\tilde{R}_i^{\im}$
are the real part and the imaginary part of $\tilde{R}_i$ respectively.
Each row of $\mathcal{C}(\tilde{B})$ must be in the form of 
$\mathcal{C}([\tilde{R}_i^T,\tilde{R}_j^T]^T) $ for some $i\neq j$.
It is not hard to check that
\[
  \mathcal{C}([\tilde{R}_i^T,\tilde{R}_j^T]^T)\in 
  \text{Row}\big(\mathcal{C}([(\tilde{R}_i^{\re})^T,(\tilde{R}_j^{\re})^T,
  (\tilde{R}_i^{\im})^T,(\tilde{R}_j^{\im})^T]^T)\big).
\]
The $\tilde{R}_i^{\re},\tilde{R}_j^{\re},\tilde{R}_i^{\im},\tilde{R}_j^{\im}$
are rows of $\hat{B} $. Thus $\mathcal{C}([\tilde{R}_i^T,\tilde{R}_j^T]^T)$ 
is in $\text{Row}(\mathcal{C}(\hat{B}))$ as well. 
It implies that $\text{Row}(\mathcal{C}(\tilde{B}))\subset \text{Row}(\mathcal{C}(\hat{B}))$.
Similarly, each row of $\mathcal{C}(\hat{B})$ 
must be in the form of $\mathcal{C}([\hat{R}_i^T,\hat{R}_j^T]^T) $ for some $i\neq j$.
For each row $\hat{R}_i$, there is a corresponding row $\tilde{R}_{i'}$ of $\tilde{B}$
such that $\hat{R}_i=\tilde{R}^{\re}_{i'}$ or $\hat{R}_i=\tilde{R}^{\im}_{i'}$.
One can check that
\[
  \mathcal{C}([\hat{R}_i^T,\hat{R}_j^T]^T)\in
  \text{Row}\big(\mathcal{C}([\tilde{R}_{i'}^T,
  \overline{\tilde{R}^T_{i'}},\tilde{R}_{j'}^T,
  \overline{\tilde{R}^T_{j'}}]^T)\big).
\]
The $\tilde{R}_{i'},\tilde{R}_{j'},\overline{\tilde{R}}_{i'},\overline{\tilde{R}}_{j'}$
are rows of $\tilde{B}$. Thus $\mathcal{C}([\hat{R}_i^T,\hat{R}_j^T]^T)$
is in the row space of $\mathcal{C}(\tilde{B})$. 
It implies that $\text{Row}(\mathcal{C}(\hat{B}))\subset \text{Row}(\mathcal{C}(\tilde{B}))$.
Therefore, $\text{Row}(\mathcal{C}(\tilde{B}))=\text{Row}(\mathcal{C}(\hat{B}))$.
By the same argument, it also holds that 
$\text{Row}(\mathcal{C}(\tilde{C}))=\text{Row}(\mathcal{C}(\hat{C}))$.
There exist nonsingular matrcies $U_B,U_C$ such that
$\mathcal{C}(\tilde{B})=U_B\mathcal{C}(\hat{B})$ and 
$\mathcal{C}(\tilde{C})=U_C\mathcal{C}(\hat{C})$. So,
\[
\mathcal{C}(\tilde{B})\odot \mathcal{C}(\tilde{C}) =
\mathcal{C}(U_B\hat{B})\odot \mathcal{C}(U_C\hat{C}) =
(U_B\boxtimes U_C)(\mathcal{C}(\hat{B})\odot \mathcal{C}(\hat{C})).
\]
Note that $U_B\boxtimes U_C$ is nonsingular since $U_B,U_C$ are nonsingular.
Therefore, $\mathcal{C}(\tilde{B})\odot \mathcal{C}(\tilde{C})$
has linearly independent columns if and only if
$\mathcal{C}(\hat{B})\odot \mathcal{C}(\hat{C})$ has linearly independent columns.
By Theorem~2.5 in \cite{de2006link}, $\mathcal{C}(\hat{B})\odot \mathcal{C}(\hat{C})$
has linearly independent columns for generic 
$\hat{B} \in \R^{n_2^2\times s},\hat{C}\in \R^{n_3^2\times s}$ when
\[
\frac{s(s-1)}{2} \le \frac{n_2^2(n_2^2-1)}{2}\frac{n_3^2(n_3^2-1)}{2} .
\]
Then we have $\mathcal{C}(\tilde{B})\odot \mathcal{C}(\tilde{C})$
has linearly independent columns for generic Hermitian matrices
$\{B_i\}_{i=1}^s,\{C_i\}_{i=1}^s $.
The set of all psd Hermitian matrices has positive Lebesgue measure
in Hermitian matrices space. Therefore,
$\mathcal{C}(\tilde{B})\odot \mathcal{C}(\tilde{C})$
has linearly independent columns for generic psd Hermitian matrices 
$\{B_i\}_{i=1}^s,\{C_i\}_{i=1}^s$.

Theorem~\ref{thm:r>n2} implies that if $A_i, B_i, C_i$ 
are generic in $\bH_+^{n_1},  \bH_+^{n_2}, \bH_+^{n_3}$ respectively, 
then $\rank(\bT(\cH)) = s $ and $\bT(\cH) $ 
has the unique rank decomposition when $s$ is in the required range.
\end{proof}

\begin{remark}
Theorem~\ref{prop:r>n2} only discusses the case $m=3$.
For $m > 3$, we can proceed similarly.
Suppose  $n_1\ge n_2 \ge \cdots \ge n_m$ and
$\cH\in \sS_\C^{[n_1,\ldots,n_m]}$ is given as
\[
\begin{array}{c}
\cH \, = \, \sum_{i=1}^r[u_i^1, \ldots, u_i^m]_{\otimes h}.
\end{array}
\]
We can flatten $\cH$ to a cubic order tensor in
$\C^{n_1^2\times n_2^2 \times M^2} $, with $M = n_3n_4\cdots n_m$.
Theorem~\ref{prop:r>n2} can then be
equivalently applied to the new tensor.
\end{remark}

When the rank $r$ is low,
the tensor $\bT(\cH)$ has a unique rank decomposition.
If this decomposition is given by \eqref{Eq:DecompositionT} for psd matrices $B_j^i$,
then $\cH$ is separable and we can also get psd decompositions.

We now present some examples illustrating this.
To compute tensor decompositions for $\bT(\cH)$ 
we use the software \texttt{Tensorlab}~\cite{tensorlab}.
The computation is implemented in MATLAB R2019b,
on an \verb!Intel(R) Core(TM) i7-8550U CPU!
with $3.79$~GHz and $16$ GB of RAM.

\begin{example}
Let $\cH \in \sS_\C^{[8,8,8]} $ be given by the psd decomposition
\[
\mathfrak{m}(\cH) = A\boxtimes B \boxtimes I_8+B\boxtimes I_8 \boxtimes A+ I_8\boxtimes A \boxtimes B,
\]
where $I_8$ is the $8\times 8$ identity matrix
and the matrices $A,B$ are given by
\[
A_{ij} = \left \{
\begin{array}{ll}
  7+i^2 & \text{if $i=j$}  \\
  6+(j-i)\sqrt{-1} & \text{if $i\neq j$}
\end{array},\right.
B_j^i = \sum_{k=1}^8 \exp(\frac{k-1}{8}(i-j)\pi\sqrt{-1}).
\]
The tensor $\cH$ satisfies the condition of
Theorem~\ref{Thm:uniqueCpsdDecomposition-smallS}.
We use \texttt{Tensorlab} to decompose $\bT(\cH)$
and successfully obtain the above psd decomposition.
The computation took about $0.2$ second.
\end{example}

\subsection{The case $m=2$}
\label{Subsec:m=2}

When $m=2$, the tensor $\bT(\cH)$ has order $2$, i.e., it is a matrix.
Matrix decompositions are never unique, unless the rank is one.
Thus, the previous uniqueness results for tensor decompositions
are not applicable for the case $m=2$.
However, we can use a different tensor flattening for Hermitian tensors.

Suppose $n_1\ge n_2$ and $\mathcal{H}\in \mathbb{C}^{[n_1,n_2]}$
is expressed as ($\lmd_i \in \R$)
\begin{equation}\label{eq: sep H}
    \cH \,=\, \lmd_1 [u_1,v_1]_\oh + \cdots + \lmd_r [u_r,v_r]_\oh.
\end{equation}
We define the new tensor flattening $\bT_1(\cH)$
on $\mathbb{C}^{[n_1,n_2]}$ such that
\[
\bT_1(\cH) = \lmd_1 u_1\otimes \overline{u}_1 \otimes (v_1\boxtimes \overline{v}_1)
+ \cdots  + \lmd_r u_r\otimes \overline{u}_r \otimes (v_r\boxtimes \overline{v}_r).
\]
Note that $\cH$ is separable if and only if $\bT_1(\cH)$
has the decomposition
\begin{equation}\label{eq:T2H}
\bT_1(\cH) = \sum_{i=1}^{R} a_i\otimes \overline{a}_i \otimes \text{vec}(B_i),
\end{equation}
for $a_i \in \mathbb{C}^{n_1}$ and $B_i \in \bH_+^{n_2}$.
The minimum $R$ in \eqref{eq:T2H} may not be the psd rank of $\cH$.
It is only an upper bound, i.e.,  $R \ge \psdrank(\cH) $.
When $R$ is small, the decomposition \eqref{eq:T2H}
can be obtained from the tensor decomposition of $\bT_1(\cH)$.

\begin{theorem} \label{prop:m=2}
Let $n_1\ge n_2 \ge 2$. Suppose $\cH \in \sS_\C^{[n_1,n_2]}$ is given
as in the decomposition \eqref{eq:T2H}. Assume that
\begin{equation} \label{rel:Rln1:m=2}
R \le \max(n_1,n_2^2), \qquad
\frac{R(R-1)}{2} \le \frac{l(l-1)}{2} \cdot \frac{n_1(n_1-1)}{2},
\end{equation}
where $l:=\min(n_1,n_2^2)$.
If all $a_i$ are generic in $\C^{n_1}$ and $B_i$ are generic in $\bH_+^{n_2}$,
then $\rank(\bT_1(\cH))=R$ and $\bT_1(\cH)$ has the unique rank decomposition
as in \eqref{eq:T2H}, up to scaling and permutation.
\end{theorem}
\begin{proof}
The proof is quite similar to that for Theorem~\ref{prop:r>n2}.
When $R$ satisfies the relation~\eqref{rel:Rln1:m=2},
the conditions for Theorem~\ref{thm:r>n2} are satisfied for generic
$a_i \in \C^{n_1}$ and generic $B_i \in \bH^{n_2}$.
Therefore, we have $\rank(\bT_1(\cH))=R $ and the rank decomposition is unique,
up to scaling and permutation.
\end{proof}

If the flattening tensor $\bT_1(\cH)$ has a small rank $R$, say, $R \le l$,
then the rank decomposition of $\bT_1(\cH)$ can be computed
(e.g., by {\tt Tensorlab}).
For the case $R > l$, computing the rank decomposition of $\bT_1(\cH)$ is harder.
If the rank decomposition of $\bT_1(\cH)$ is in the form \eqref{eq:T2H},
then $\cH$ must be separable.

\begin{remark}
Let $\mathcal{H}\in \sS_\C^{[n_1,n_2]}$
be the separable Hermitian tensor as in \eqref{eq: sep H}, where $n_1\ge n_2$.
In addition to $\bT_1(\cH)$, we may also introduce another new tensor
\[
		\bT_2(\cH) := (u_1\boxtimes \overline{u}_1)\otimes v_1\otimes \overline{v}_1
+ (u_r\boxtimes \overline{u}_r)\otimes v_r\otimes \overline{v}_r
\]
Similarly, $\cH$ is separable if and only if $\bT_2(\cH)$ has the decomposition
\[
\begin{array}{c}
\bT_2(\cH) \, = \, \sum_{i=1}^R \Vect(A_i)\otimes b_i \otimes \tilde{b}_i ,
\end{array}
\]
where $A_i\in \bH^{n_1}_+$. Equivalently, if
\[
\begin{array}{c}
R\le n_1^2, \quad
\frac{R(R-1)}{2} \le \frac{n_2(n_2-1)}{2}\frac{n_2(n_2-1)}{2},
\end{array}
\]
then $\rank(\bT_2(\cH))=R$ and the above decomposition is unique for generic
$b_i \in \C^{n_2}$ and $A_i \in \bH_+^{n_1}$.
For some cases, it may give a better upper bound than Theorem~\ref{prop:m=2}.
Therefore, we can also use the rank decomposition of $\bT_2(\cH)$
to certify separability of Hermitian tensors.
\end{remark}

We conclude this section with some examples showcasing our methods for the specific case $m=2$.

\begin{example}
Consider the Hermitian tensor $\cH\in \sS_\C^{[4,3]}$ such that
\[
\cH_{i_1i_2j_1j_2} = (4i_1j_1+\sqrt{-1}(j_1-i_1)+1)\delta_{i_2j_2}+i_1j_1(i_2-j_2)\sqrt{-1}
\]
where $\delta_{ij} = 1$ for $i=j$ and $\delta_{ij} = 0$ otherwise.
\texttt{Tensorlab} yields the following tensor decomposition for $\bT_1(\cH)$:
{\scriptsize
\[
%\bT_1(\cH) =
\begin{pmatrix}
    1 \\ 2 \\3 \\4
  \end{pmatrix}\otimes
 \begin{pmatrix}
    1 \\ 2 \\3\\ 4
\end{pmatrix} \otimes \text{vec}
 \begin{pmatrix}
    3 & -\sqrt{-1} & -2\sqrt{-1} \\
     \sqrt{-1} & 3  & -\sqrt{-1}  \\
     2\sqrt{-1} & \sqrt{-1} & 3
\end{pmatrix}
+
\begin{pmatrix}
    1+\sqrt{-1} \\ 2+\sqrt{-1} \\3+\sqrt{-1}\\4+\sqrt{-1}
\end{pmatrix}\otimes
 \begin{pmatrix}
    1-\sqrt{-1} \\ 2-\sqrt{-1} \\3-\sqrt{-1}\\4-\sqrt{-1}
\end{pmatrix} \otimes
 \text{vec}(I_3),
\]}\noindent
where $I_3$ denotes the $3\times 3$ identity matrix.
The above decomposition certifies that $\cH$ is separable.
The computation took about 0.1 second.
\end{example}

\section{Conclusions and discussions}
\label{Sec:Conclusion}

This paper studies how to detect separability of Hermitian tensors
and how to compute psd decompositions.
The Algorithm~\ref{Alg:membership-in-sep-cone} is given to
decide whether a given Hermitian tensor is separable.
It is based on solving Lasserre type moment relaxations.
Its asymptotic and finite convergence are proved under certain conditions.
If a given Hermitian tensor is not separable,
the algorithm can detect the non-separability;
if it is separable, the algorithm yields a positive Hermitian decomposition.
Consequently, this settles the ``quantum separability problem'' in quantum physics.
We remark that Algorithm~\ref{Alg:membership-in-sep-cone} can detect separability
for all Hermitian tensors, no matter
if the dimensions/ranks are low or high.
However, this method does not guarantee that
the computed decomposition has the shortest length.
It is important future work to find the shortest
positive Hermitian decompositions for separable Hermitian tensors.

\begin{problem}
For a separable Hermitian tensor $\cH$,
how to find its shortest positive Hermitian decomposition?
\end{problem}

Furthermore, the paper discusses the psd decompositions and psd ranks
for separable Hermitian tensors.
If the Hermitian tensor has a certain low psd rank,
we are able to get psd decompositions,
based on uniqueness results of tensor decompositions.
This gives a second approach to certify separability of Hermitian tensors.
The first approach, i.e., Algorithm~\ref{Alg:membership-in-sep-cone}
based on solving moment optimization,
is able to detect separability for all Hermitian tensors.
However, in practice, it is limited to relatively small sized Hermitian tensors,
since it requires to solve a hierarchy of semidefinite relaxations.
The second approach, i.e., the methods in Section~\ref{Sec:C-psd-decomp}
based on unique tensor decompositions and tensor flattenings,
can detect separability for larger sized Hermitian tensors.
However, it only provides a sufficient condition for separability,
but it cannot detect non-separability.
Along this way, we made a good start on determine
psd ranks of separable Hermitian tensors.
Although these results shed some light on the challenging question of determining
psd ranks, much work is left to be done for fully solving the question.
Determining psd decompositions and psd ranks for more general cases
are important future work.

We would like to remark that the methods proposed in the paper
can also be applied to detect the $\R$-separability.
To be $\R$-separable, a Hermitian tensor $\cH$ must be $\R$-Hermitian decomposable
(i.e., \eqref{Eq:rank-oneHD} holds for real vectors $u_i^j$).
It is interesting to note that $\R$-Hermitian decomposable tensors
are $\R$-separable if and only if they are $\C$-separable
(see \cite[Lemma~6.2]{NYHerm20}).
Therefore, the $\R$-separability of $\R$-Hermitian decomposable tensors
can also be detected by checking their $\C$-separability.
In fact, detecting $\R$-separable is easier than detecting $\C$-separability,
since the imaginary part vectors $x_j^{\Im}$ are all zero
for $\R$-separable Hermitian tensors.
Similar moment optimization problems \eqref{program:sep-meas},
\eqref{program:sep-momcone} and their semidefinite relaxations for
\eqref{program:sep-SDP}, \eqref{program:sep-SDP-dual}
can be formulated for $\R$-separability.
For $m = 2$, we refer to the work \cite{NieZhang16}.
For $m > 2$, similar work can be done.

\end{document}